\documentclass[10pt,twoside]{article}
\usepackage{amsmath, amsthm, amsfonts, amssymb}
\usepackage{mathrsfs}
\usepackage[misc]{ifsym}
\usepackage[colorlinks,linkcolor=blue,anchorcolor=blue,citecolor=blue,urlcolor=black]{hyperref}
\usepackage{graphicx}
\usepackage{subfigure}
\usepackage{tikz}
\usepackage{enumerate}
\usepackage{anysize}
\usepackage{setspace}
\usepackage{float}
\usepackage{color}
\usepackage{chngcntr}
\usepackage[numbers,sort&compress]{natbib}
\everymath{\displaystyle}
\allowdisplaybreaks

\newfam\msbfam
\font\tenmsb=msbm5    \textfont\msbfam=\tenmsb \font\sevenmsb=msbm5
\scriptfont\msbfam=\sevenmsb \font\fivemsb=msbm5
\scriptscriptfont\msbfam=\fivemsb

\newfam\bigfam
\font\tenbig=msbm5 scaled \magstep2   \textfont\bigfam=\tenbig
\font\sevenbig=msbm7 scaled \magstep2 \scriptfont\bigfam=\sevenbig
\font\fivebig=msbm5 scaled \magstep2
\scriptscriptfont\bigfam=\fivebig

\numberwithin{equation}{section}
\newtheorem{theorem}{Theorem}[section]
\newtheorem{lemma}{Lemma}[section]

\newtheorem{remark}{Remark}[section]

\newtheorem{definition}{Definition}[section]
\begin{document}
	\title{\bf  New Multilinear Littlewood--Paley  $g_{\lambda}^{*}$ Function and Commutator on  Weighted Lebesgue Spaces}
	\author{\bf Huimin Sun, Shuhui  Yang and Yan Lin$^*$}
	\renewcommand{\thefootnote}{}
	\date{}
	\maketitle
	\footnotetext{This work was partially supported by the National 
		Natural Science Foundation of China No. 12071052.}
	\footnotetext{$^*$Corresponding author, E-mail:linyan@cumtb.edu.cn}
	\begin{minipage}{13.5cm}
		{\bf Abstract}
		\quad
		Via  the new weight function $A_{\vec p}^{\theta }(\varphi )$, the authors  introduce a new class of multilinear Littlewood--Paley $g_{\lambda}^{*}$ functions and establish the boundedness on weighted Lebesgue spaces. In addition, the authors obtain the boundedness of the multilinear commutator and multilinear iterated commutator generated by the multilinear Littlewood--Paley $g_{\lambda}^{*}$ function and the new $BMO$ function on weighted Lebesgue spaces. The results in this article include the known results in \cite{XY2015,SXY2014}. When $m=1$, that is, in the case of one-linear, our conclusions are also new, further extending the results in \cite{S1961}.
	\end{minipage}
	
\section{Introduction}\label{sec1}
  \quad\quad Weighted inequalities play an important role in harmonic analysis by providing important estimates for nonlinear partial differential equations and vector-valued inequalities. Recall that the classical Muckenhoupt’s $A_{p}$ weight functions originated in Muckenhoupt’s work in \cite{M1972}. Many scholars have shown great interest in this topic. We refer the reader also to \cite{CMP2004,GR1985,D2000,G1983,J1986,R1984} for some studies on the Muckenhoupt’s $A_{p}$ weight functions. 
    
Let $\omega(x) \in L_{loc}^1(\mathbb{R}^n)$ and be nonnegative. A weight $\omega$ is said to be of class  
$A_p$ with $1<p<\infty$, if and only if 
there exists a positive constant $C$ such that for any cube $Q$,
\begin{align*}
	\left(\frac{1}{|Q|}\int_Q\omega(y)dy\right)^{\frac{1}{p}}
	\left(\frac{1}{|Q|}\int_Q\omega(y)^{\frac{-1}{p-1}}dy\right)^
	{\frac{1}{p'}}\leq C,
	\end {align*}
	where $p'$ satisfies $ 1/p+1/p'=1$.
	When $ p=1 $, $\omega \in A_1$ if and only if there exists a positive constant 
	$C$ such that
	\begin{align*}
	\left(\frac{1}{|Q|}\int_Q\omega(y)dy\right)\leq C\mathop{\inf}\limits_{x\in 
	{\rm Q}}\omega(x).
	\end {align*}
  	
  Based on the classical Muckenhoupt’s class weight functions, many authors have studied the extrapolation of weights. $A_{p}$ extrapolation theorem was first proved by Rubio de Francia in \cite{R1984}. Later, the $A_{p}$ extrapolation theorem is further reduced and generalized by Cruz-Uribe \cite{CMP2004} et al. in 2004. 
  In 2012, Tang \cite{T2012} derived weighted norm inequalities for smooth symbolic pseudo-differential operators and their commutators by introducing a new weighted function $A_{p}^{\infty}(\varphi)$. For some other studies of $A_{p}^{\infty}(\varphi)$, one can see \cite{GZ2019,HZ2018,LZCZ2016,T2014} and related references.
  	
  	We say that $\omega \in A_p^{\theta}(\varphi)$ with $ 1< p< \infty $  and $\theta\geq0$, if and only if there exists a 
  	positive constant $C$ such that
  	\begin{align*}
  	\left(\frac{1}{\varphi(Q)^{\theta}|Q|}\int_Q\omega(y)dy\right)^{\frac{1}{p}}
  	\left(\frac{1}{\varphi(Q)^{\theta}|Q|}\int_Q\omega(y)^{\frac{-1}{p-1}}dy\right)^
  	{\frac{1}{p'}}\leq C,
  	\end {align*}
  	for all cubes $Q$. Especially, when $p=1$, we obtain
  	\begin{align*}
  	\left(\frac{1}{\varphi(Q)^{\theta}|Q|}\int_Q\omega(y)dy\right)\leq 
  	C\mathop{\inf}\limits_{x\in {\rm Q}}\omega(x),
  	\end {align*}
  	where $Q=Q(x,r)$ is a cube with $x$ as the center and $r$ as the sidelength,  $\varphi(Q):=1+r$, $|E|$ denotes the Lebesgue measure of $E$ and $\chi_E$ represents the characteristic function of $E$. Moreover, $f_Q$ is used to denote the average, given by $f_Q:=\frac{1}{|Q|} \int_Qf(y)dy$.
  	\begin{remark}
  		Let $A_p^{\infty}(\varphi):= {\textstyle \bigcup_{\theta \ge 0}} 
  		A_p^{\theta}(\varphi)$ and $A_{\infty}^{\infty}(\varphi):= {\textstyle \bigcup_{p\ge 1}} A_p^{\infty}(\varphi)$. The class $A_p^{\theta}(\varphi)$ is 
  		strictly larger than the class $A_p$ for all $1\leq p<\infty$. Particularly, when $ \theta =0$, $A_p^0(\varphi)$ is equivalent to the Muckenhoupt's class $A_p$ for all $1\leq p<\infty$. Moreover, for $ \omega (x):=1+|x|^\gamma $, $ \gamma>n(p-1)$, then $\omega (x)\in A_p^{\infty}(\varphi)$, but $\omega (x)\notin A_{p}$.
  	\end{remark}
  		
   In \cite{1GT2002}, Grafakos and Torres raised the question:“Is there a multiple weight theory? The most appropriate multilinear maximal function and multiple weights to work with in this direction have not been yet clear.” In 2009, Lerner \cite{LOP2009} et al. gave the following definition of multiple weights in order to answer the question. Let $1\leq p_{j} < \infty$, $j=1,\cdots ,m$, and $1/p=1/p_{1}+\ldots+1/p_{m}$. Given $\vec{\omega}=(\omega_1,\cdots,\omega_m)$, set $v_{\vec{\omega}}=\prod_{j=1}^{m}\omega_{j}^{\frac{p}{p_{j}}}$, $\vec{\omega}\in A_{\vec{p}}$ if and only if
    \begin{align*} 
	\sup_{Q}\left(\frac{1}{|Q|}\int_{ Q}\nu_{\vec\omega}(y)dy\right)^{\frac{1}{p}}\prod_{j=1}^{m}\left(\frac{1}{|Q|}\int_{ Q}\omega_{j}^{1-p_{j}'}(y)dy\right)^{\frac{1}{p_{j}'}}< \infty,
    \end{align*}
where the supremum is taken over all cubes $Q\subset \mathbb{R}^n$. When 
$p_j=1$, $j=1,\ldots, m$, the term $\left(\frac{1}{|Q|}\displaystyle \int_Q 
\omega_j(y)^{1-p'_j}dy\right)^{\frac{1}{p'_j}}$ is considered to be 
$\left(\mathop{\inf}\limits_{y\in {\rm Q}}\omega_j(y)\right)^{-1}$.

 In recent years, scholars have increasingly focused their attention on weighted inequalities. In \cite{B2015}, Bui et al. gave the following definition of $A_{\vec{p}}^{\infty}(\varphi)$. Other studies on $A_{\vec{p}}^{\infty}(\varphi)$ include some related literature (see, for instance, \cite{1ZZ2021,BHS2011,HZ2023,YLL,LYL}). 

We refer to $\vec\omega\in A_{\vec{p}}^{\theta}(\varphi)$, if and only if
\begin{align*}
	\mathop{\sup}\limits_{\rm Q} 
	\left(\frac{1}{\varphi(Q)^{\theta}|Q|}\int_Qv_{\vec{\omega}}(x)dx\right)^{\frac1p}
	\prod_{j=1}^{m}\left(\frac{1}{\varphi(Q)^{\theta}|Q|}\int_Q\omega_j(x)^{1-p'_j}dx\right)^{\frac{1}{p'_j}}<\infty,
	\nonumber
\end{align*}
where the supremum is taken over all cubes $Q\subset \mathbb{R}^n$ and  
$\theta\geq0$. When $p_j=1$, $j=1,\ldots, m$, the term 
$\left(\frac{1}{|Q|}\displaystyle \int_Q  
\omega_j(x)^{1-p'_j}dx\right)^{\frac{1}{p'_j}}$ is considered to be 
$\left(\mathop{\inf}\limits_{x\in {\rm Q}}\omega_j(x)\right)^{-1}$.   
 
	\begin{remark}
	Denote $A_{\vec{p}}^{\infty}(\varphi):= {\textstyle \bigcup_{\theta \ge 0}} 
	A_{\vec{p}}^{\theta}(\varphi)$ for $1\leq p_{j}< \infty$, $j=1,\cdots ,m$. When $\theta=0$, the class 
	$A_{\vec{p}}^{0}(\varphi)$ is considered to be the class of multiple weights 
	$A_{\vec{p}}$ introduced by {\rm \cite{LOP2009}}.
	\end{remark}

 Now, we recall some background of multilinear operators. The development of the multilinear Calder\'{o}n--Zygmund theory can be traced back to the works of Coifman and Meyer in the 1970s (see, for instance, \cite{1CM1978,2CM1978,CM1975}). The multilinear Calder\'{o}n--Zygmund operator with standard kernel was reserached by Grafakos et al.(see, for instance, 
 \cite{2GT2002,3GT2002}). Furthermore, there have been several studies pertaining to multilinear Calder\'{o}n--Zygmund operators with nonstandard kernels (see, for instance, \cite{DJ1984,LZ2014,Y1985,ZS2019}).
 
  With the study of the new class of multiple weights $A_{\vec{p}}^{\infty}(\varphi)$, multilinear Calder\'{o}n--Zygmund operators are further studied. In \cite{PT2015}, Pan and Tang successfully obtained a class of multilinear Calder\'{o}n--Zygmund operators with classical kernels on weighted Lebesgue spaces with weights in $A_{\vec{p}}^{\infty}(\varphi)$. They also established the boundedness of the iterated commutators generated by the new $BMO$ functions and these operators. In \cite{1ZZ2021}, Zhao and Zhou proved the boundedness of a class of operators with classical kernels on Morrey spaces with the weight of $A_{\vec{p}}^{\theta}(\varphi)$. The boundedness of commutators and iterated commutators on Morrey spaces was also obtained.
 At the same time, Zhao and Zhou \cite{2ZZ2021} introduced a class of multilinear Calder\'{o}n--Zygmund operators with Dini's type kernels. The strong and weak boundedness on weighted Lebesgue spaces was also proved. 
  
The study of multilinear theory has developed significant attention due to the rapid development of harmonic analysis. In 1978, Coifman and Meyer \cite{2CM1978} introduced the Littlewood--Paley $g$-functions. In the past few decades, many scholars have done a series of studies on Littlewood–Paley $g$-functions. Recently, with the study of new weights of $A_{\vec{p}}^{\infty}(\varphi)$, the Littlewood--Paley $g$-functions is further studied. In \cite{LYL,YLL}, Li, Yang and Lin introduced a new class of  multilinear square operators with generalized kernels and studied the boundedness of these operators and multilinear commutators on weighted Lebesgue spaces and weighted Morrey sapces.

 As we all know, the sublinear operators $g_{\lambda}^{*}$ are a variant of classical Littlewood--Paley $g$-functions, which play a very important role in harmonic analysis and PDE. In 1961, Stein \cite{S1961} first defined classical $g_{\lambda}^{*}$ operators with higher dimensions and proved that if $\lambda\textgreater2$ , then $g_{\lambda}^{*}$ is of weak type $(1, 1)$ and is of strong type $(p, p)$ for $1\textless p \textless \infty$. In 1970,  Fefferman \cite{F1970} made some other estimates of $g_{\lambda}^{*}$ functions. In 2002, Ding \cite{DLY2002} et al. gave the $L^{p}$-boundeness of $g_{\lambda}^{*}$ functions. In 2007, Xue \cite{XDY2007} et al. gave the weighted weak $(1, 1)$ boundedness and the weighted $L^{p}$ boundedness for the Littlewood--Paley $g_{\lambda}^{*}$ function with complex parameters. Then, many scholars have also studied the $g_{\lambda}^{*}$ function on different function spaces (see, for instance, \cite{L2011,C2023}).

In recent years, the study of multilinear Littlewood--Paley   $g_{\lambda}^{*}$ functions has also been widely concerned. In 2014, Shi, Xue and Yabuta \cite{SXY2014} showed that if the kernel $K$ is in the form of convolution type, $K(x-y_{1} \ldots x-y_{m})$, then the $g_{\lambda}^{*}$ operator is bounded from $L^{1}\times \cdots \times L^{1}$ to $L^{\frac{1}{m},\infty}$, and has the strong weighted estimates $L^{p_1}(\omega_1)\times \cdots \times L^{p_m}(\omega_m)$ to $L^{p_{m}}(\nu_{\vec\omega})$ and  weighted weak type $L^{p_1}(\omega_1)\times \cdots \times L^{p_m}(\omega_m)$ to $L^{p,\infty}(\nu_{\vec\omega})$. In 2015, Xue \cite{XY2015} defined the $T_{\lambda}(\vec{f})$, established the  $L^{p_1}(\omega_1)\times \cdots \times L^{p_m}(\omega_m)$ to $L^{p_{m}}(\nu_{\vec\omega})$ and the weak type $L^{p_1(\omega_1)}\times \cdots \times L^{p_m}(\omega_m)$ to $L^{p,\infty}(\nu_{\vec\omega})$ estimates of $T_{\lambda}(\vec{f})$. If the kernel $K$ is a multilinear Littlewood--Paley kernel, we denote  $T_{\lambda}$ by $g_{\lambda}^{*}$.

The new multilinear Littlewood--Paley operator with classical kernel is defined as follows. Let  $ K_t(z,y_1,\ldots,y_m) $ be a locally integrable function defined away from the diagonal $z=y_1=\ldots=y_m$ in $ (\mathbb{R}^{n})^{m} $. It is said to be a classical kernel, if for some positive constants $ C $, $\gamma  $, $B\textgreater1 $ and for any $ N\ge 0 $,  it satisfies the following size condition for any $\lambda>0$.
\begin{equation}\label {1.1}	
	\left(\iint_{\mathbb{R}_{+}^{n+1}}\left(\frac{t}{|x-z|+t}\right)^{n\lambda}|K_{t}(z,\vec{y})|^{2}\frac{dzdt}{t^{n+1}}\right)^{\frac{1}{2}}\\
	\leq\frac{C}{\left(\sum_{j=1}^{m}|x-y_{j}|\right)^{mn}\left(1+\sum_{j=1}^{m}|x-y_{j}|\right)^{N}},
\end{equation}
and the smoothness condition
\begin{equation} \label {1.2}
	\begin{aligned}
	\left(\iint_{\mathbb{R}_{+}^{n+1}}\left(\frac{t}{|z|+t}\right)^{n\lambda}|K_{t}(x-z,\vec{y})-K_{t}(x'-z,\vec{y})|^{2}\frac{dzdt}{t^{n+1}}\right)^{\frac{1}{2}}\\
	\leq\frac{C|x-x'|^{\gamma}}{\left(\sum_{j=1}^{m}|x-y_{j}|\right)^{mn+\gamma}\left(1+\sum_{j=1}^{m}|x-y_{j}|\right)^{N}},
  \end{aligned}
\end{equation}
whenever $|x-x'|\leq\frac{1}{B}\max_{1\leq j \leq m}{|x-y_j|}$.

Suppose that $\omega(t):[0,\infty) \mapsto  [0,\infty)$ is a nondecreasing function with $0<\omega(t)<\infty$. One says that $\omega\in {\rm Dini}(a)$ for $a>0$, if
\begin{align*}
	[\omega]_{{\rm Dini}[a]}:=\int^{1}_{0}\frac{\omega^a(t)}{t}dt<\infty.  
\end{align*}
In addition, if $0<a_1<a_2$, Dini($a_1$) $\subset$ Dini($a_2$). 

The new operator with kernel of Dini's type is defined as follows, 
with the size condition \eqref {1.1} remaining unchanged and the smoothness condition \eqref {1.2} being changed by
\begin{equation}\label{1.3}
	\begin{aligned}
		&\left(\iint_{\mathbb{R}_{+}^{n+1}}\left(\frac{t}{|z|+t}\right)^{n\lambda}|K_{t}(x-z,\vec{y})-K_{t}(x'-z,\vec{y})|^{2}\frac{dzdt}{t^{n+1}}\right)^{\frac{1}{2}}\\
		&\leq\frac{C|x-x'|^{\gamma}}{\left(\sum_{j=1}^{m}|x-y_{j}|\right)^{mn+\gamma}\left(1+\sum_{j=1}^{m}|x-y_{j}|\right)^{N}}\omega 
		\left( \frac{|x-x'|}{\sum_{j=1}^m |x-y_{j}|}\right),
	\end{aligned}
\end{equation}
whenever $|x-x'|\leq\frac{1}{B}\max_{1\leq j \leq m}{|x-y_j|}$.

The conditions \eqref{1.3} can be further weakened, then we can obtain a class of generalized kernels.
\begin{equation}\label{1.4}	
	\begin{aligned}
		&\left(\iint_{\mathbb{R}_{+}^{n+1}}\left(\frac{t}{|z|+t}\right)^{n\lambda}\left(\int_{\left(\Delta_{k+2}\right)^{m}\backslash\left(\Delta_{k+1}\right)^{m}}\left|K_{t}(x-z,\vec{y})-K_{t}(x'-z,\vec{y})\right|^{q}d\vec{y} \right)^{\frac{2}{q}}\frac{dzdt}{t^{n+1}} \right)^{\frac{1}{2}}\\\\
		&\leq CC_k2^{\frac{-kmn}{q'}}|x-x'|^{\frac{-mn}{q'}}(1+2^k|x-x'|)^{-N},
	\end{aligned}
\end{equation}
where $\Delta_k:=Q(x', 2^{k}\sqrt{mn}|x-x'|)$ represents the cube with center at $x'$ and side length $ 2^{k}\sqrt{mn}|x-x'| $,
$k\in \mathbb{N}$, $ 
1< q\le 2$, $1/q+1/q'=1$ and  $C_k>0$. 

\begin{definition}\label{definition1.1}
	Let $K(y_0,y_1,\ldots,y_m)$ be a locally integrable function defined away 
	from the diagonal $y_0=y_1=\ldots=y_m$ in $(\mathbb{R}^n)^{m+1}$ and $ 
	K_t(z,y_1,\ldots ,y_m):=\left(\frac{1}{t^{mn}}\right)K\left(\frac{z}{t},\frac{y_1}{t},\ldots,\frac{y_m}{t}\right)$, for any $ 
	t\in (0,\infty )$ and  $ \lambda>0 $. The multilinear Littlewood--Paley  $g_{\lambda}^{*}$ function is defined by
	\begin{equation}\label{1.5}
		g_{\lambda}^{*}(\vec{f})(x):=\left(\iint_{\mathbb{R}_{+}^{n+1}}\left(\frac{t}{|x-z|+t}\right)^{n\lambda}\left|\int_{(\mathbb{R}^{n})^{m}}K_{t}(z,\vec{y})\prod_{j=1}^{m}f_{j}(y_{j})d\vec{y}\right|^{2} \frac{ dzdt}{t^{n+1}}\right)^{\frac{1}{2}}
	\end{equation}
	
	\noindent whenever $x\notin  {\textstyle \bigcap_{j=1}^{m}}  {\rm supp } f_j$ 
	and each $ f_{j} \in C_{c}^{\infty}\left(\mathbb{R}^{n}\right) $.
	
	If the following conditions are satisfied, $g_{\lambda}^{*}$  is denoted as the new multilinear Littlewood--Paley function with generalized kernel.
	
\begin{itemize}
	\item [\rm (i)] The kernel function satisfies the conditions \eqref{1.1} and \eqref{1.4}	 for any $N\geq 0$.
	
	\item [\rm (ii)]  $g_{\lambda}^{*}$ can be extended to be a 
	bounded operator from $L^{s_1}(\mathbb{R}^{n})\times \cdots \times L^{s_m}(\mathbb{R}^{n})$ to $L^{s,\infty}(\mathbb{R}^{n})$ 
	for some $ 1\le s_j\le q', j=1,\ldots,m $ with	$\frac{1}{s}=\frac{1}{s_1}+\ldots+\frac{1}{s_m}$.
\end{itemize}
	
\end{definition}

\begin{remark}
	When $ N=0 $ in the condition \eqref{1.1} and \eqref{1.2}, the kernel returns to a standard kernel which has been studied by Xue in {\rm \cite{XY2015}}.
\end{remark}

\begin{remark}
	When $ \omega (t):=t^\gamma $ for some $ \gamma > 0 $ in the condition 
	\eqref{1.3}, the new multilinear Littlewood--paley function with kernel of Dini's type is 
	the multilinear Littlewood--paley function with classical kernel.
\end{remark}

\begin{remark}
   Since the class of generalized kernels represent a further weakening of Dini's type, it follows that they can encompass Dini’s type. In fact, 
	when $C_k:=\omega(2^{-k})$, it is clear that the condition 
	\eqref{1.3} implies the condition \eqref{1.4}, for any $1< q\le 2$.
\begin{align*}
	&\left(\iint_{\mathbb{R}_{+}^{n+1}}\left(\frac{t}{|z|+t}\right)^{n\lambda}\left(\int_{\left(\Delta_{k+2}\right)^{m}\backslash\left(\Delta_{k+1}\right)^{m}}\left|K_{t}(x-z,\vec{y})-K_{t}(x'-z,\vec{y})\right|^{q}d\vec{y} \right)^{\frac{2}{q}}\frac{dzdt}{t^{n+1}} \right)^{\frac{1}{2}}\\
	&\leq \left(\int_{\left(\Delta_{k+2}\right)^{m}\backslash\left(\Delta_{k+1}\right)^{m}}\left(\iint_{\mathbb{R}_{+}^{n+1}}\left(\frac{t}{|z|+t}\right)^{n\lambda} \left|K_{t}(x-z,\vec{y})-K_{t}(x'-z,\vec{y})\right|^{q\cdot\frac{2}{q}}\frac{dzdt}{t^{n+1}}  \right)^{\frac{q}{2}}d\vec{y}\right)^{\frac{1}{q}}\\
	&\leq\left(\int_{\left(\Delta_{k+2}\right)^{m}\backslash\left(\Delta_{k+1}\right)^{m}}\left[\frac{|x-x^{'}|^{\gamma}}{(\sum_{j=1}^{m}|x-y_{j})^{mn+\gamma}(1+(\sum_{j=1}^{m}|x-y_{j})^{N})}\omega\left(\frac{|x-x'|}{\sum_{j=1}^{m}|x-y_{j}|}\right)\right]^{q}d\vec{y}\right)^{\frac{1}{q}}\\
	&\leq \left(\int_{(\Delta_{k+2})^m \setminus 
		(\Delta_{k+1})^m}\left[\frac{C}{(2^k \sqrt{n}|x-x'|)^{mn}(1+2^k 
		\sqrt{n}|x-x'|)^N}\omega \left( \frac{1}{2^k}\right)\right]^q 
	d\vec{y}\right)^{1/q}\\
	&\leq 
	CC_k2^{\frac{-kmn}{q'}}|x-x'|^{\frac{-mn}{q'}}(1+2^k|x-x'|)^{-N}.
\end{align*}
\end{remark}

In \cite{BHS2011}, B. Bongioanni et al. gave the following definition of the new $BMO$ space,  which encompasses the classical $BMO$ space in 2011. Gradually, many researchers have established weighted norm inequalities for multilinear commutators with $BMO_{\theta}(\varphi)$ functions(see, for instance, \cite{PT2015,2ZZ2021}).

\begin{definition}{\rm(\cite{BHS2011})}\label{definition1.2}
	The new $ BMO $ space $  BMO_{\theta  }(\varphi ) $ with $ \theta \ge 0 $  is 
	defined as a set of all locally integrable functions $ b $ satisfying,
	\begin{align*}
		\Vert 
		b\Vert_{BMO_{\theta}(\varphi)}:=\mathop{\sup}\limits_Q\frac{1}{\varphi(Q)^{\theta}
			|Q|}\int_Q|b(y)-b_Q|dy<\infty.
		\nonumber
	\end{align*}
	When $ \theta=0 $, 
	$BMO_\theta(\varphi)=BMO(\mathbb{R}^n)$. $BMO(\mathbb{R}^n) 
	\subset BMO_{\theta}(\varphi)$ and 
	$BMO_{{\theta}_1}(\varphi) \subset BMO_{{\theta}_2}(\varphi)$ for 
	${\theta}_1\leq {\theta}_2$. We define $BMO_{\infty}(\varphi):=  {\textstyle 
		\bigcup_{\theta \ge 0}} BMO_{\theta}(\varphi)$.
\end{definition}
Let $ \vec \theta =(\theta _1,\dots ,\theta _m) $ with $ \theta _1,\dots 
,\theta _m\ge 0 $ and $ \vec b=(b_1,\dots ,b_m) $. If $ b_j\in BMO_{\theta 
	_j}(\varphi ) $ for $ 1\le j\le m $, then $ \vec b\in BMO_{\vec \theta }^{m} 
(\varphi ) $.

The article is organized as follows. In Section \ref{sec2}, we give some definitions, symbols and establish the boundedness of the new multilinear Littlewood--Paley $g_{\lambda}^{*}$ operators with generalized kernels on weighted Lebesgue spaces. The boundedness of multilinear commutator generated by a new class of  multilinear Littlewood--Paley functions and new $ BMO $ functions on weighted Lebesgue spaces will be showed in Section \ref{sec3}. In Section  \ref{sec4}, we establish the boundedness of multilinear iterative commutators generated by a new class of  multilinear Littlewood--Paley functions and new $ BMO $ functions on weighted Lebesgue spaces.
\section{New Weighted Norm Inequalities for  $g_{\lambda}^{*}$ Operators
	with  Generalized Kernels on weighted Lebesgue Spaces}\label{sec2}

\quad\quad In Subsection 2.1, we provide certain definitions and symbols that will be used later. In Subsection 2.2, we establish the pointwise estimate for the sharp maximal function. In Subsection 2.3, the boundness of $g_{\lambda}^{*}$ operators with  generalized kernels on weighted Lebesgue spaces is given. 
\subsection{Definitions and Lemmas}
\begin{definition}{\rm(\cite{LZCZ2016})}
	The dyadic maximal function $M_{\varphi,\eta}^\triangle$  for $0<\eta<\infty$, $f\in L_{loc}^{1}$ is defined by
	\begin{align*}
	M_{\varphi,\eta}^\triangle f(x):=\mathop{\sup}\limits_{{\rm Q}\ni x }\frac{1}{\varphi(Q)^{\eta}|Q|}\int_Q|f(y)|dy,
	\end{align*}
where $ Q $ is any dyadic cube containing $x$ in $\mathbb{R}^{n}$.

	The dyadic sharp maximal operator $M_{\varphi,\eta}^{\sharp,\triangle}$ is 
	defined by 
	\begin{align*}
	M_{\varphi,\eta}^{\sharp,\triangle} f(x)
	&:=\mathop{\sup}\limits_{{\rm Q}\ni x ,r<1}\frac{1}{|Q|}\int_{Q}|f(y)-f_Q|dy+\mathop{\sup}\limits_{{\rm Q}\ni x,r\geq1}\frac{1}{\varphi(Q)^{\eta}|Q|}\int_{Q}|f(y)|dy\\
	&\simeq\mathop{\sup}\limits_{{\rm Q}\ni x,r<1}\mathop{\inf}\limits_{\mathit c }\frac{1}{|Q|}\int_{Q}|f(y)-c|dy+\mathop{\sup}\limits_{{\rm Q}\ni x,r\geq1}\frac{1}{\varphi(Q)^{\eta}|Q|}\int_{Q}|f(y)|dy,
	\end{align*}
	where $f_{Q} :=\frac{1}{\left | Q \right | } \displaystyle \int_Qf(y)dy$ and $ Q $ is the dyadic cube. 
\end{definition}
\begin{lemma}\label{Lemma2.1}{\rm(\cite{BHS2011})}  
	The following statements hold.
	\begin{itemize}
	\item [\rm (i)] $A_p^{\infty}(\varphi) \subset A_q^{\infty}(\varphi) $, $1\leq 
	p<q<\infty$.
	\item [\rm (ii)] If $\omega \in A_p^{\infty}(\varphi )$ with $p>1$, then there 
	exists a $\varepsilon>0$ such that $\omega \in  
	A_{p-\varepsilon}^{\infty}(\varphi )$. Consequently, 
	$A_p^{\infty}(\varphi ):= {\textstyle \bigcup_{q< p}} A_q^{\infty}(\varphi )$.
	\item [\rm (iii)] If $\omega \in A_p^{\infty}(\varphi )$ for $p\geq1$, then there exist positive numbers $l$, $\delta$ and $C$ such that for every cube $Q$,
	\begin{align*}
	\left(\frac{1}{|Q|}\int_Q\omega^{1+\delta}(x)dx\right)^{\frac{1}{1+\delta}}\leq
	C\left(\frac{1}{|Q|}\int_Q\omega(x)dx\right)(1+r)^l.
	\end{align*}
	\end{itemize}
\end{lemma}


\begin{lemma}\label{Lemma2.2}{\rm(\cite{B2015})}  
	 The following statements are equivalent for  $1\leq p_1,\ldots,p_m<\infty$ and $\vec{\omega}=(\omega_1,\ldots,\omega_m)$.
	\begin{itemize}
	\item [\rm (i)] $\omega_j^{1-p'_j} \in A_{mp'_j}^{\infty}$, for
	$j=1,\ldots,m $ and $v_{\vec{\omega}} \in A_{mp}^{\infty}(\varphi )$.
	\item [\rm (ii)] $\vec\omega \in A_{\vec{p}}^{\infty}(\varphi ).$
	\end{itemize}
	\end {lemma}
	\begin{remark}
		The $A_{\vec{p}}^{\infty}(\varphi )$ weight function is not increasing, implying that for $\vec{p}=(p_1,\ldots,p_m)$ and $\vec{q}=(q_1,\ldots,q_m)$ with $p_j<q_j$, $j=1,\ldots, m$, it may not be the case that $A_{\vec{p}}^{\infty}(\varphi ) \subset A_{\vec{q}}^{\infty}(\varphi )$.
	\end{remark}
	\begin{lemma}\label{Lemma2.3}{\rm(\cite{T2012})}
		Let $1<p<\infty$, $0<\eta<\infty$ and $\omega\in A_p^\infty(\varphi )$, the following statements hold for $f\in L^p(\omega)$,
		\begin{align*}
		\Vert f \Vert_{L^p(\omega)}\leq\Vert M_{\varphi,\eta}^\triangle f 
		\Vert_{L^p(\omega)}\leq C\Vert M_{\varphi,\eta}^{\sharp, \triangle} f 
		\Vert_{L^p(\omega)}.
		\end{align*}
	
	For $0<\delta < \infty$, denote as
	\begin{align*}
		M^\triangle_{\delta,\varphi,\eta}f(x):=[M^\triangle_{\varphi,\eta}
		(|f|^\delta)]^{\frac1\delta}(x), M^{\sharp, 
		\triangle}_{\delta,\varphi,\eta}f(x):=[M^{\sharp, 
		\triangle}_{\varphi,\eta}(|f|^\delta)]^{\frac1\delta}(x).
	\end{align*}
	\end {lemma}
	\begin{lemma} \label{Lemma2.4}{\rm(\cite{T2012})} 
		Let $\psi: (0,\infty)\to(0,\infty)$ be doubling, meaning there exists a constan $C_0$ such that for $a>0$, $\psi(2a)\leq C_0\psi(a)$. Under these conditions, there exists a constant $C>0$ that depends on the doubling  condition of  $\psi$ and the $ 
		A_\infty^\infty(\varphi )$ condition of  $\omega$ for $0<\eta<\infty$, $\omega\in A_\infty^\infty(\varphi )$ and 
		$\delta>0$ such that,
   \begin{align*}
		\mathop{\sup}\limits_{\lambda>0}\psi(\lambda)\omega\left(\lbrace 
		y\in\mathbb{R}^n, M_{\delta.\varphi,\eta}^\triangle f(y)>\lambda\rbrace\right)
		\leq C \mathop{\sup}\limits_{\lambda>0}\psi(\lambda)\omega\left(\lbrace 
		y\in\mathbb{R}^n, M_{\delta.\varphi,\eta}^{\sharp, \triangle} 
		f(y)>\lambda\rbrace\right).
	\end{align*}
		\end {lemma}
		
	The multilinear maximal operators $\mathcal{M}_{\varphi,\eta}$ and $ \mathcal{M}_{\delta,\varphi,\eta} $ for $0<\eta<\infty$, $0<\delta < \infty$ and $\vec{f}=(f_1,\ldots, f_m)$, $f_j\in L_{loc}^1(\mathbb{R}^n)$ are defined by,
	\begin{align*}
		\mathcal{M}_{\varphi,\eta}(\vec{f})(x):=\mathop{\sup}\limits_{x\in {\rm	Q}}\prod_{j=1}^m\frac1{\varphi(Q)^\eta \left | Q \right | } \int_Q 
		|f_j(y_j)| dy_j,
	\end{align*}
	and
	\begin{align*}
	\mathcal{M}_{\delta,\varphi,\eta}(\vec f)(x)
	:=\left [ \mathcal{M}_{\varphi,\eta}(|\vec{f}|^{\delta })(x) \right ] 
	^{\frac{1}{\delta } }
	=\left(\mathop{\sup}\limits_{x\in 
	{\rm Q}}\prod_{j=1}^m\frac1{\varphi(Q)^\eta \left | Q \right |} \int_Q 
	|f_j(y_j)|^\delta dy_j \right)^{\frac1\delta}.
	\end{align*}
		
	\begin{lemma}\label{Lemma2.5}{\rm(\cite{PT2015})}
		Let $\vec \omega\in A_{\vec{p}}^\infty(\varphi )$ and $1<p_j<\infty$, 
		$j=1,\ldots, m$, $1/p = {\textstyle \sum_{j=1}^{m}}  1/p_j
		$, then there exists some $\eta_0>0$ depending on $p, m , p_j$, such that
	\begin{align*}
	\Vert \mathcal{M}_{\varphi,\eta_0}(\vec{f}) \Vert_{L^{p}(v_{\vec \omega })}\leq 
	C\prod_{j=1}^{m}\Vert f_j \Vert_{L^{p_j}(\omega_j)}.
	\end{align*}
	\end {lemma}
			
    \begin{lemma}\label{Lemma2.6}{\rm(\cite{PT2015})}
	Let $\vec \omega\in A_{\vec p}^{\infty } (\varphi )$ and $1\leq p_j<\infty$, 
	$1/p = {\textstyle \sum_{j=1}^{m}}  1/p_j$, 
	$j=1,\ldots, m$, then there exists some $\theta_0>0$ depending on $p, m, p_j$, 
	such that
	\begin{align*}
	\Vert \mathcal{M}_{\varphi,\theta_0}(\vec{f}) \Vert_{L^{p,\infty}(v_{\vec 
	\omega })}\leq C\prod_{j=1}^{m}\Vert f_j \Vert_{L^{p_j}(\omega_j)}.
	\end{align*}
	\end {lemma}
	\begin{lemma}\label{Lemma2.7}{\rm(\cite{LOP2009})}
	Suppose $0<p<q<\infty$. Then there is a 
	positive constant $C:=C_{p,q}$, such that for all measurable functions $f$,
	\begin{align*}
	|Q|^{-1/p}\Vert f \Vert_{L^p (Q)}\leq C|Q|^{-1/q}\Vert f 
	\Vert_{L^{q,\infty}(Q)}.
	\end{align*}
	\end{lemma}
	
	\subsection{ Sharp Maximal Pointwise Estimates}
	\quad\quad In this subsection, we provide a pointwise estimate for the sharp maximal function of the multilinear Littlewood--Paley $g_{\lambda}^{*}$ operators by other classes of maximal functions. The following is the main result of this subsection.
	\begin{theorem}\label{Theorem 2.1} 
	Suppose that $g_{\lambda}^{*}$ is the new multilinear Littlewood--Paley function with generalized kernel as in Definition \ref{definition1.1} and $\sum_{k=1}^{\infty}C_{k}<\infty$. If $0<\delta<1/m$ and $0<\eta<\infty$, then there exists a constant $C>0$ such that 
	\begin{align*}
	M_{\delta,\varphi,\eta}^{\sharp, \triangle}(g_{\lambda}^{*}(\vec{f}))(x)\leq 
	C\mathcal{M}_{q',\varphi,\eta}(\vec f)(x)
	\end{align*}
 for all bounded measurable functions $\vec 
f=(f_1,\ldots,f_m)$ with compact support.
\end{theorem}
\begin{proof}
\noindent{\it \rm} 
Fix a pioint $x\in\mathbb{R}^n$ and for any dyadic cube $ Q:=Q(x_0,r)\ni x$, to prove the Theorem \ref{Theorem 2.1}, we examine two cases about $ r $: $r<1$ and $r\geq1$.

\noindent{\it \rm \textbf{Case 1}:} $r<1$. Since $0<\delta<1/m< 1$, $\eta>0$ 
and $\left | \left | A_{1}  \right | ^{t} -  \left | A_{2}  \right | ^{t} 
\right | \le \left | A_{1}- A_{2}  \right | ^{t} $ for $0<t<1$, for any 
constant $A$, we have
\begin{align*}
	\left(\frac1{|Q|}\int_Q \left||g_{\lambda}^{*}(\vec f)(\mu)|^\delta-|A|^\delta \right| 
	d\mu\right)^{\frac1\delta}\leq
	\left(\frac1{|Q|}\int_Q \left|g_{\lambda}^{*}(\vec f)(\mu)-A \right|^\delta 
	d\mu\right)^{\frac1\delta}.
\end{align*}

Let $Q^*:=14n\sqrt{nm}Q$, we split each $f_j$ as 
$f_j=f_j^0+f_j^\infty=f_j \chi_{Q^*}+f_j \chi_{(Q^*)^{c}}$. Write
\begin {align*}
\prod_{j=1}^mf_j(y_j)
&=\prod_{j=1}^m f_j^{0}(y_j)+\sum_{(\alpha_1,\ldots,\alpha_m)\in 
	\mathscr{L}}f_1^{\alpha_1}(y_1)\cdots f_m^{\alpha_m}(y_m),
\end {align*}
where  $\mathscr{L}:=\{(\alpha_1,\ldots,\alpha_m)$: there is at least one $\alpha_j=\infty\}$. 

Take $\mu_0\in 4Q\setminus 3Q$ and

	$A:=\left(\iint_{\mathbb{R}_{+}^{n+1}}\left(\frac{t}{|z|+t}\right)^{n\lambda}\left|\int_{({\mathbb{R}^n})^m}K_{t}(\mu_{0}-z,\vec{y})\sum_{(\alpha_1,\ldots,\alpha_m)\in\mathscr{L}}f_{1}^{\alpha_{1}}(y_{1})\cdots f_{m}^{\alpha_{m}}(y_{m})d\vec{y}\right|^{2}\frac{dzdt}{t^{n+1}}\right)^{\frac{1}{2}}.$
 
 Then for any $\mu\in Q$,
 \begin  {align*}
 &| g_{\lambda}^{*}(\vec{f})(\mu)-A|
\leq \Bigg(\iint_{\mathbb{R}_{+}^{n+1}}\left(\frac{t}{|z|+t}\right)^{n\lambda}\bigg|\int_{({\mathbb{R}^n})^m}\bigg[K_{t}(\mu-z,\vec{y})\prod_{j=1}^{m}f_{j}(y_{j})\\
 &\quad\quad\quad\quad\quad\quad-K_{t}(\mu_{0}-z,\vec{y})\sum_{(\alpha_1,\ldots,\alpha_m)\in\mathscr{L}}f_{1}^{\alpha_{1}}(y_{1}) \cdots f_{m}^{\alpha_{m}}(y_{m})\bigg]d\vec{y}\bigg|^{2}\frac{dzdt}{t^{n+1}}\Bigg)^{\frac{1}{2}}\\
 &\quad\quad\quad\quad\quad\leq\Bigg(\iint_{\mathbb{R}_{+}^{n+1}}\left(\frac{t}{|z|+t}\right)^{n\lambda}\bigg|\int_{({\mathbb{R}^n})^m}K_{t}(\mu-z,\vec{y})\prod_{j=1}^{m}f_{j}^{0}(y_{j})d\vec{y}\bigg|^{2}\frac{dzdt}{t^{n+1}}\Bigg)^{\frac{1}{2}}\\
  &\quad\quad\quad\quad\quad\quad+\sum_{(\alpha_1,\ldots,\alpha_m)\in\mathscr{L}}\Bigg(\iint_{\mathbb{R}_{+}^{n+1}}\left(\frac{t}{|z|+t}\right)^{n\lambda}\bigg|\int_{({\mathbb{R}^n})^m}\big[K_{t}(\mu-z,\vec{y})-K_{t}(\mu_{0}-z,\vec{y})\big]\\
  &\quad\quad\quad\quad\quad\quad\times f_{1}^{\alpha_{1}}(y_{1}) \cdots f_{m}^{\alpha_{m}}(y_{m})d\vec{y}\bigg|^{2}\frac{dzdt}{t^{n+1}}\Bigg)^{\frac{1}{2}}.
 \end {align*}
 
 Thus
\begin  {align*}
&\left(\frac1{|Q|}\int_Q | g_{\lambda}^{*}(\vec{f})(\mu)-A|^\delta d\mu
\right)^{\frac1\delta}\\
&\leq C\left(\frac1{|Q|}\int_Q |g_{\lambda}^{*}(f_1^0,\ldots,f_m^0)(\mu)|^\delta 
d\mu\right)^{\frac1\delta}\\
&\quad +C\sum_{(\alpha_1,\ldots,\alpha_m)\in \mathscr{L}}\Bigg(\frac{1}{|Q|}\int_{ Q}\bigg(\iint_{\mathbb{R}_{+}^{n+1}}\left(\frac{t}{|z|+t}\right)^{n\lambda}\bigg|\int_{({\mathbb{R}^n})^m}\big[K_{t}(\mu-z,\vec{y})-K_{t}(\mu_{0}-z,\vec{y})\big]\\
&\quad\times f_{1}^{\alpha_{1}}(y_{1}) \cdots f_{m}^{\alpha_{m}}(y_{m})d\vec{y}\bigg|^{2}\frac{dzdt}{t^{n+1}}\bigg)^{\frac{\delta}{2}}d\mu\Bigg)^{\frac{1}{\delta}}\\
&:=\uppercase\expandafter{\romannumeral+1}+\uppercase\expandafter{\romannumeral+2}\\
&:=\uppercase\expandafter{\romannumeral+1}+\sum_{(\alpha_1,\ldots,\alpha_m)\in \mathscr{L}}\uppercase\expandafter{\romannumeral+2}_{\alpha_1,\ldots,\alpha_m}.
\end {align*}
By $g_{\lambda}^{*}:L^{s_1}\times\cdots\times L^{s_m}\to L^{s,\infty}$, Lemma 
\ref{Lemma2.7} and H\"{o}lder's inequality, we can get that
\begin {align*}
\uppercase\expandafter{\romannumeral+1}
&\leq C|Q|^{-\frac{1}{s}}\Vert g_{\lambda}^{*}(f_1^0,\ldots,f_m^0) \Vert_{L^{s,\infty}}
\leq 
C\prod_{j=1}^m\left(\frac1{|Q^*|}\int_{Q^*}|f_j(y_j)|^{s_j}dy_j\right)^{\frac1{s_j}}\\
&\leq 
C\prod_{j=1}^m\left(\frac1{|Q^*|}\int_{Q^*}|f_j(y_j)|^{q'}dy_j\right)^{\frac1{q'}}
\leq 
C\prod_{j=1}^m\left(\frac1{\varphi(Q^*)^\eta|Q^*|}\int_{Q^*}|f_j(y_j)|^{q'}dy_j\right)^{\frac1{q'}}\\
&\leq C\mathcal{M}_{q',\varphi,\eta}(\vec f)(x).
\end {align*}

Take   
$\Delta_k:=Q(\mu_0, 2^{k}\sqrt{mn}|\mu-\mu_0|)$, $  \mu\in Q $, $k\in \mathbb{N}_+$. Note that  $\Delta_2 \subset Q^*$, so 
$(\mathbb{R}^n)^m\setminus 
(Q^*)^m\subset (\mathbb{R}^n)^m\setminus (\Delta_2)^m$. When $\mu\in Q$, $\mu_0\in 
4Q\setminus 3Q$, we can infer that $|\mu-\mu_0|\sim r$. Since each term of 
$\mathscr{L}$ contains at least one $\alpha_j=\infty, j=1, \ldots,m$, we select 
one of them for discussion. Choosing $N\geq m\eta/q'$, by H\"{o}lder's inequality, Minkowski's inequality and \eqref{1.4}, we have
\begin{align*}
	\uppercase\expandafter{\romannumeral+2}_{\alpha_1,\ldots,\alpha_m}
	&\leq \frac{1}{|Q|}\int_{ Q}\Bigg(\iint_{\mathbb{R}_{+}^{n+1}}\left(\frac{t}{|z|+t}\right)^{n\lambda}\Big(\int_{(\mathbb{R}^{n})^{m}\backslash (Q^\ast)^{m}}\left|K_{t}(\mu-z,\vec{y})-K_{t}(\mu_{0}-z,\vec{y})\right|\\
	 &\quad \times\prod_{j=1}^{m}\left|f_{j}(y_{j})\right|d\vec{y}\Big)^{2} \frac{    dzdt}{t^{n+1}}\Bigg)^{\frac{1}{2}}d\mu\\
	&\leq \frac{1}{|Q|}\int_{ Q}\sum_{k=1}^{\infty}\Bigg(\iint_{\mathbb{R}_{+}^{n+1}}\left(\frac{t}{|z|+t}\right)^{n\lambda}\Big(\int_{\left(\Delta_{k+2}\right)^{m}\backslash\left(\Delta_{k+1}\right)^{m}}\left|K_{t}(\mu-z,\vec{y})-K_{t}(\mu_{0}-z,\vec{y})\right|\\
	 &\quad \times\prod_{j=1}^{m}\left|f_{j}(y_{j})\right|d\vec{y}\Big)^{2} \frac{    dzdt}{t^{n+1}}\Bigg)^{\frac{1}{2}}d\mu\\	
	&\leq \frac{1}{|Q|}\int_{ Q}\sum_{k=1}^{\infty}\Bigg(\iint_{\mathbb{R}_{+}^{n+1}}\left(\frac{t}{|z|+t}\right)^{n\lambda}\Big(\int_{\left(\Delta_{k+2}\right)^{m}\backslash\left(\Delta_{k+1}\right)^{m}}\big|K_{t}(\mu-z,\vec{y})\\
	&\quad-K_{t}(\mu_{0}-z,\vec{y})\big|^{q}d\vec{y}\Big)^{\frac{2}{q}}\frac{dzdt}{t^{n+1}}\Bigg)^{\frac{1}{2}} \prod_{j=1}^{m} \Big( \int_{\Delta_{k+2}} 
	\left|f_{j}\left(y_{j}\right)\right|^{q'}  d{y_{j}} \Big) ^\frac{1}{q'}d\mu\\
	&\leq \frac{C}{|Q|}\int_Q \sum_{k=1}^\infty C_k 
	2^{\frac{-kmn}{q'}}|\mu-\mu_0|^{\frac{-mn}{q'}}(1+2^k|\mu-\mu_0|)^{-N}  
	\prod_{j=1}^m\left(\int_{\Delta_{k+2}}|f_j(y_j)|^{q'} dy_j 
	\right)^{\frac 1{q'}} d\mu\\
	&\leq \frac{C}{|Q|}\int_Q \sum_{k=1}^\infty \frac{C_k 
	\varphi(\Delta_{k+2})^{\frac{m\eta}{q'}}|\Delta_{k+2}|^{\frac{m}{q'}}}{(1+2^kr)^N(2^kr)^{\frac{mn}{q'}}}
	\prod_{j=1}^m\left(\frac{1}{\varphi(\Delta_{k+2})^{\eta}|\Delta_{k+2}|}\int_{\Delta_{k+2}}|f_j(y_j)|^{q'}dy_j\right)^{\frac 1{q'}}d\mu\\
	&\leq C\mathcal{M}_{q',\varphi,\eta}(\vec f)(x).
\end{align*}
Thus, $\uppercase\expandafter{\romannumeral+2}\leq C\mathcal{M}_{q',\varphi,\eta}(\vec f)(x)$.

\noindent{\it \rm \textbf{Case 2}:} $r\geq1$. Let $\widetilde{Q}^*:=8Q$, we split each $f_j$ as 
$f_j=f_j^0+f_j^\infty=f_j \chi_{\widetilde{Q}^*}+f_j \chi_{(\widetilde{Q}^*)^{c}}$. Write
\begin {align*}
\prod_{j=1}^mf_j(y_j)
&=\prod_{j=1}^m f_j^{0}(y_j)+\sum_{(\alpha_1,\ldots,\alpha_m)\in 
	\mathscr{L}}f_1^{\alpha_1}(y_1)\cdots f_m^{\alpha_m}(y_m),
\end {align*}
where  $\mathscr{L}:=\{(\alpha_1,\ldots,\alpha_m)$: there is at least one $\alpha_j=\infty\}$.

 Since $0<\delta<\frac{1}{m}< 1$ and 
$\eta>0$, we obtain
\begin {align*}
&\left(\frac1{\varphi(Q)^\eta|Q|}\int_Q |g_{\lambda}^{*}(\vec f)(\mu)|^\delta 
d\mu\right)^{\frac1\delta}\\
&\leq C\frac{1}{\varphi(Q)^{\eta/\delta}}\left(\frac1{|Q|}\int_Q 
|g_{\lambda}^{*}(f_1^0,\ldots,f_m^0)(\mu)|^\delta d\mu\right)^{\frac1\delta}\\
& +C\frac{1}{\varphi(Q)^{\eta/\delta}}\sum_{\alpha_1,\ldots,\alpha_m\in 
	\mathscr{L}}\left(\frac1{|Q|}\int_Q|g_{\lambda}^{*}(f_1^{\alpha_1},\ldots,f_m^{\alpha_m})(\mu)|^\delta
d\mu\right)^{\frac1\delta}\\
&:=\uppercase\expandafter{\romannumeral+3}+\uppercase\expandafter{\romannumeral+4}\\
&:=\uppercase\expandafter{\romannumeral+3}+\sum_{\alpha_1,\ldots,\alpha_m\in \mathscr{L}}\uppercase\expandafter{\romannumeral+4}_{\alpha_{1},\ldots,\alpha_m}.
\end {align*}

By Lemma \ref{Lemma2.7}, $g_{\lambda}^{*}:L^{s_1}\times\cdots\times L^{s_m}\to L^{s,\infty}$ and H\"{o}lder's inequality, we conclude that
\begin {align*}
\uppercase\expandafter{\romannumeral+3}
&\leq C\frac{1}{\varphi(Q)^{m\eta}}|Q|^{-\frac{1}{s}}\Vert g_{\lambda}^{*}(f_1^0,\ldots,f_m^0) 
\Vert_{L^{s,\infty}}
\leq 
C\frac{1}{\varphi(Q)^{m\eta}}\prod_{j=1}^m\left(\frac1{|Q^{*}|}\int_{Q^{*}}|f_j(y_j)|^{s_j}dy_j\right)^{\frac1{s_j}}\\
&\leq 
C\frac{1}{\varphi(Q)^{m\eta}}\prod_{j=1}^m\left(\frac1{|Q^*|}\int_{Q^{*}}|f_j(y_j)
|^{q'}dy_j\right)^{\frac1{q'}}
\leq 
C\prod_{j=1}^m\left(\frac1{\varphi(Q^{*})^\eta|Q^{*}|}\int_{Q^{*}}|f_j(y_j)|^{q'}dy_j\right)^{\frac1{q'}}\\
&\leq C\mathcal{M}_{q',\varphi,\eta}(\vec f)(x).
\end {align*}

 It is easy to get that $ {\textstyle \sum_{j=1}^{m}} |\mu-y_j|\sim 2^{k}r$ for $\mu\in Q$ and $(y_1,y_2,\ldots,y_{m})\in (2^{k+3}Q)^m \setminus (2^{k+2}Q)^m$. Since $\mu\in Q$, $\mu_0\in 4Q\setminus 3Q$, we can obtain that $|\mu-\mu_0|\sim r$. Taking $N >m\eta/q'$, by Minkowski's inequality and the size condition \eqref{1.1}, so
\begin{align*}
	\uppercase\expandafter{\romannumeral+4}_{\alpha_1,\ldots,\alpha_m}&\leq \frac C{|Q|}\int_Q| g_{\lambda}^{*}(f_1^{\alpha_1},\ldots,f_m^{\alpha_m})(\mu)|d\mu\\
	&\leq\frac{C}{|Q|}\int_{ Q}\Bigg(\iint_{\mathbb{R}_{+}^{n+1}}\left(\frac{t}{|\mu-z|+t}\right)^{n\lambda}\Big(\int_{{({\mathbb{R}^n})^m}\backslash\left({\widetilde Q}^*\right)^{m}}\left|K_{t}(z,\vec{y})\right|\\
	&\quad\times \prod_{j=1}^{m}\left|f_{j}(y_{j})\right|d\vec{y}\Big) ^{2} \frac{dzdt}{t^{n+1}}\Bigg)^{\frac{1}{2}}d\mu\\
	&\leq C\frac{1}{|Q|}\int_{ Q}\int_{{({\mathbb{R}^n})^m}\backslash\left({\widetilde Q}^*\right)^{m}}\left(\iint_{\mathbb{R}_{+}^{n+1}}\left(\frac{t}{|\mu-z|+t}\right)^{n\lambda}|K_{t}(z,\vec{y})|^{2}\frac{dzdt}{t^{n+1}}\right)^{\frac{1}{2}}\prod_{j=1}^{m}|f_{j}(y_{j})|d\vec{y}d\mu\\
	&\leq C\frac{1}{|Q|}\int_{ Q}\sum_{k=1}^{\infty}\int_{\left(2^{k+3}Q\right)^{m}\backslash\left(2^{k+2}Q\right)^{m}}\frac{1}{\left(\sum_{j=1}^{m}|\mu-y_{j}|\right)^{mn}(1+\sum_{j=1}^{m}|\mu-y_{j}|)^{N}}\\
	&\quad\times\prod_{j=1}^{m}|f_{j}(y_{j})|d\vec{y}d\mu\\
	&\leq C\frac{1}{|Q|}\int_{ Q}\sum_{k=1}^{\infty}\frac{|2^{k+3}Q|^{m}}{(2^{k}r)^{mn}(1+2^{k}r)^{N}}\prod_{j=1}^{m}\left(\frac{1}{|2^{k+3}Q|}\int_{2^{k+3}Q }|f_{j}(y_{j})|^{q'}dy_{j}\right)^{\frac{1}{q'}}d\mu\\
	&\leq C\frac{1}{|Q|}\int_{ Q}\sum_{k=1}^{\infty}\frac{|2^{k+3}Q|^{m}\varphi(2^{k+3}Q)^{\frac{m\eta}{q'}}}{(2^{k}r)^{mn}(1+2^{k}r)^{N}}\prod_{j=1}^{m}\left(\frac{1}{\varphi(2^{k+3}Q)^{\eta}|2^{k+3}Q|}\int_{2^{k+3}Q }|f_{j}(y_{j})|^{q'}dy_{j}\right)^{\frac{1}{q'}}d\mu\\
	&\leq C\frac{1}{|Q|}\int_{ Q}\sum_{k=1}^{\infty}\left(1+2^{k}r\right)^{\frac{m\eta}{q'}-N}\prod_{j=1}^{m}\left(\frac{1}{\varphi(2^{k+3}Q)^{\eta}|2^{k+3}Q|}\int_{2^{k+3}Q }|f_{j}(y_{j})|^{q'}dy_{j}\right)^{\frac{1}{q^{'}}}d\mu\\
	&\leq C\mathcal{M}_{q',\varphi,\eta}(\vec f)(x).
\end{align*}
Thus, $\uppercase\expandafter{\romannumeral+4}\leq 
C\mathcal{M}_{q',\varphi,\eta}(\vec f)(x)$.

Then, we complete the proof of the Theorem \ref{Theorem 2.1}.
\end{proof}
\subsection{The Boundedness of the $g_{\lambda}^{*}$ on weighted Lebesgue Spaces }
\quad\quad In this subsection, with the help of Theorem \ref{Theorem 2.1}, the boundedness of multilinear Littlewood--Paley $g_{\lambda}^{*}$ operator with generalized kernel on weighted Lebesgue spaces are obtained. The main theorem of this subsection is given below.
\begin{theorem}\label{Theorem 2.2} 
Suppose that $m\geq2$,  $g_{\lambda}^{*}$ is the new multilinear Littlewood--Paley function defined in  Definition \ref{definition1.1} with generalized kernel, $ {\textstyle 
	\sum_{k=1}^{\infty }}  C_{k} < \infty $,  $\vec \omega=(\omega_1,\ldots ,\omega_m) \in 
A_{\vec{p}/q'}^\infty(\varphi )$, $v_{\vec\omega}= {\textstyle \prod_{j=1}^{m}} \omega_j^{p/p_j}$, $ \vec p=(p_1, \ldots ,p_m) $ and $1/p = {\textstyle \sum_{j=1}^{m}}  1/p_j$. Then the following results hold.
\begin{itemize}
	\item [\rm (i)] If $q'<p_j<\infty$, $j=1,\ldots,m$, then there exists a constant $C>0$ such that
	\begin{align*}
		\Vert g_{\lambda}^{*}(\vec{f}) \Vert_{L^{p}(v_{\vec\omega})}\leq 
		C\prod_{j=1}^{m}\Vert f_j \Vert_{L^{p_j}(\omega_j)}.
	\end{align*}
	
	\item [\rm (ii)] If $q'\leq p_j<\infty$, $j=1,\ldots,m$ and at least one of 
	$p_j=q'$, then there exists a constant $C>0$ such that
	\begin{align*}
		\Vert g_{\lambda}^{*}(\vec{f}) \Vert_{L^{p,\infty}(v_{\vec\omega})}\leq 
		C\prod_{j=1}^{m}\Vert f_j \Vert_{L^{p_j}(\omega_j)}.
	\end{align*}
\end{itemize}
\end{theorem}
\begin{proof}\rm (i) Take a $ \delta  $ such that $ 0< \delta < 1/m $. We select $ \eta =\eta _0 $ in Lemma \ref{Lemma2.5} for $ \vec \omega \in A_{\vec p/q'}^\infty(\varphi ) $. It derives from Lemma \ref{Lemma2.2} that $v_{\vec\omega}\in A_{mp/q'}^\infty(\varphi ) \subset A_{p/\delta } ^\infty(\varphi )$. Then, by Lemma \ref{Lemma2.3}, Theorem \ref{Theorem 2.1} and Lemma \ref{Lemma2.5}, we conclude that
\begin{align*}
	\| g_{\lambda}^{*}(\vec f) \|_{L^p(v_{\vec\omega})} 
	&= \| |g_{\lambda}^{*}(\vec f)|^{\delta } \| _{L^{p/\delta  
		}(v_{\vec\omega})}^{1/\delta  } 
	\leq \| M_{\varphi,\eta _0}^{\bigtriangleup }(|g_{\lambda}^{*}(\vec f)|^{\delta }) 
	\|_{L^{p/\delta }(v_{\vec\omega})}^{1/\delta }\\
	&\leq C \| M_{\delta,\varphi,\eta _0}^{\sharp, \triangle}(g_{\lambda}^{*}(\vec f)) 
	\|_{L^p(v_{\vec\omega})}
	\leq C\| \mathcal{M}_{q',\varphi,\eta _0}(\vec f) \|_{L^p(v_{\vec\omega})}\\
	&=C\| \mathcal{M}_{\varphi,\eta _0}(|\vec f|^{q'}) 
	\|_{L^{p/q'}(v_{\vec\omega})}^{1/q'}
	\leq C\prod_{j=1}^{m}\| |f_j|^{q'} 
	\|_{L^{p_j/q'}(\omega_j)}^{1/q'}\\
	&=C \prod_{j=1}^{m}\| f_j \|_{L^{p_j}(\omega_j)}.
\end{align*}
The expected result is achieved.

\rm (ii) We derive the weak type estimate for the new $ g_{\lambda}^{*}$ operator  with  generalized kernel on the weighted Lebesgue spaces for $ \theta _0 $ is taken by Lemma \ref{Lemma2.6} with $ \vec \omega \in A_{\vec p/q'}^\infty(\varphi ) $.
\begin{align*}
	\| g_{\lambda}^{*}( \vec f )   \|_{L^{p,\infty } \left ( \upsilon 
		_{\vec \omega }  \right ) }
	&=\sup_{t > 0}t \left | \nu _{\vec\omega }\left \{ x\in 
	\mathbb{R}^n: \left | g_{\lambda}^{*} (\vec f) ( x)  \right 
	|>t \right\} \right |^\frac{1}{p} \\
	&=\sup_{t > 0}\left [t^\delta  \left |\nu _{\vec\omega }\left \{ 
	x\in \mathbb{R}^n: \left |g_{\lambda}^{*} ( \vec f) (x) \right 
	|^\delta >t ^\delta \right \} \right |^\frac{\delta }{p}\right 
	]^\frac{1}{\delta } \\
	&=\sup_{t > 0}\left [t\left |\nu _{\vec\omega }\left \{ x\in 
	\mathbb{R}^n: \left |g_{\lambda}^{*} ( \vec f  ) ( x  ) \right 
	|^\delta >t   \right \} \right |^\frac{\delta }{p}\right]^\frac{1}{\delta } \\
	&=\left \|g_{\lambda}^{*}( \vec f  ) ^\delta  \right \|_{L^{\frac{p}{\delta } ,\infty } \left ( \upsilon _{\vec \omega }  \right ) }^{\frac{1}{\delta } } 
	\le \left \| {M}_{\varphi ,\theta _0}^{\bigtriangleup} \left ( \left | g_{\lambda}^{*}( \vec f  ) \right | ^\delta \right )   \right \|_{L^{\frac{p}{\delta } ,\infty } \left ( \upsilon _{\vec \omega }  \right ) }^{\frac{1}{\delta } } \\
	&=\sup_{t > 0}\left [t \left |\nu _{\vec\omega }\left \{ x\in 
	\mathbb{R}^n:M_{\varphi ,\theta _0}^{\bigtriangleup} \left ( \left |g_{\lambda}^{*} (\vec f) (x) \right |^\delta \right )  >t  \right \} \right |^\frac{\delta }{p}\right ]^\frac{1}{\delta }\\
	&=\sup_{t > 0}t \left |\nu _{\vec\omega }\left 
	\{ x\in \mathbb{R}^n: M_{\varphi ,\theta _0}^{\bigtriangleup}  \left ( \left |g_{\lambda}^{*} ( \vec f ) (x) \right |^\delta  \right ) >t ^\delta \right \} \right |^\frac{1 }{p} \\
	&=\sup_{t > 0}t \left |\nu _{\vec\omega }\left \{ x\in \mathbb{R}^n: \left (M_{\varphi ,\theta _0}^{\bigtriangleup} \left ( \left |g_{\lambda}^{*} ( \vec f ) (x) \right |^\delta \right ) \right )^\frac{1}{\delta } >t \right \} \right |^\frac{1 }{p} \\
	&:=H,\nonumber
\end{align*}

Next, we estimate $ H $. Take $ \psi (t ):=t ^{p} $, then $ 
\psi (2t )=2^{p} t ^{p} \le 2^{p}\psi (t ) $, so $ 
\psi $ satisfies the double condition. By Lemma \ref{Lemma2.4}, Lemma 
\ref{Lemma2.6} and Theorem \ref{Theorem 2.1}, we have
\begin{align*}
	H 
	&\le C \sup_{t > 0}t \left |\nu _{\vec\omega }\left 
	\{ x\in \mathbb{R}^n:M_{\delta,\varphi,\theta_0}^{\sharp,\bigtriangleup}  
	\left ( g_{\lambda}^{*} ( \vec f  ) \right )  ( x  ) >t  \right \} \right |^\frac{1 }{p} \\
	&=C\left \| M_{\delta ,\varphi ,\theta _0}^{\sharp ,\bigtriangleup} 
	\left(g_{\lambda}^{*}( \vec f  ) \right)\right \|_{L^{p,\infty } \left ( \upsilon _{\vec \omega }\right ) }\le C\left \| \mathcal{M}_{q',\varphi ,\theta _0} ( \vec f) \right \|_{L^{p,\infty } \left ( \upsilon _{\vec \omega }  \right ) } \\
	&=C \left \|\mathcal{M}_{\varphi ,\theta _0}( |\vec f|  ^{q'}) \right 
	\|_{L^{p/q' ,\infty } \left ( \upsilon _{\vec \omega }  \right ) 
	}^{1/q' } \le C\prod_{j=1}^{m} \left \||  f_j  | ^{q'} \right 
	\|_{L^{p_j/q' } ( \omega  _{j }   ) }^{1/q' }  \\
	&=C\prod_{j=1}^{m} \left \|  f_j   \right \|_{L^{p_j} ( \omega  _{j }   ) }.
	\nonumber
\end{align*}
Theorem \ref{Theorem 2.2} is thus proved.
\end{proof}
\section{The Boundedness of Multilinear Commutator on weighted Lebesgue Spaces}\label{sec3}

\quad\quad In Subsection 3.1, we give some definition and notations which will be used later. In Subsection 3.2, we establish the sharp maximal pointwise estimates. In Subsection 3.3, we prove the boundedness of multilinear commutator on weighted Lebesgue spaces.

\subsection{Definitions and Lemmas}
\begin{definition}
 The multilinear commutator $ g_{\lambda,\sum \vec{b}}^{*} $  generated by the new multilinear Littlewood--Paley operator and a family of locally integrable functions $\vec{b}=(b_1,\ldots,b_m)$ is defined by,
\begin{equation}
g_{\lambda ,\sum \vec{b}}^{*}(\vec{f})(x):=\sum_{j=1}^m g_{\lambda, b_j}^{*, j}(\vec{f})(x):=\sum_{j=1}^{m}g_{\lambda}^{*}(f_{1},\ldots,(b_{j}(x)-b_{j})f_{j},\ldots,f_{m})(x).\nonumber
\end{equation}
	
The multilinear iterative commutator $ g_{\lambda, \prod\vec{b}}^{*} $ generated by the new 
multilinear Littlewood--Paley operator is defined by,
\begin{equation}
g_{\lambda,\prod\vec{b}}^{*}(\vec f)(x):=g_{\lambda}^{*}((b_{1}(x)-b_{1})f_{1},\ldots,(b_{m}(x)-b_{m})f_{m})(x).
\nonumber
\end{equation}

\end{definition}

\begin{lemma}\label{Lemma3.1}{\rm(\cite{B2015})}  
	Let $\theta\geq0$ and $a\geq1$. If $b\in BMO_{\theta}(\varphi)$, then for cubes 
	$Q:=Q(x,r)$,
	\begin{itemize}
		\item [\rm (i)] $\left(\frac1{|Q|} \displaystyle \int_Q|b(y)-b_Q|^a dy 
		\right)^{\frac1a}\leq C\Vert b 
		\Vert_{BMO_{\theta}(\varphi)}{\varphi}(Q)^{\theta}$.
		\item [\rm (ii)] $\left(\frac1{|{t_0}^k Q|} \displaystyle 
		\int_{{t_0}^kQ}|b(y)-b_Q|^a 
		dy \right)^{\frac1a}\leq Ck\Vert b 
		\Vert_{BMO_{\theta}(\varphi)}{\varphi}({t_0}^kQ)^{\theta}$, for all $k\in 
		\mathbb{N}$ and $ t_0=2,3 $.
	\end{itemize}
	\end {lemma}
	
			
	
	\begin{lemma}\label{Lemma3.2}{\rm(\cite{B2015} )}  
		Let $1\leq p_1,\ldots,p_m<\infty$ and $\vec{\omega}=(\omega_1,\ldots,\omega_m) 
		\in A_{\vec{p}}^{\infty}(\varphi) $. Then
	\begin{itemize}
		\item [\rm (i)] $\vec{\omega}\in A_{r\vec{p}}^{\infty}(\varphi)$ for $r\geq 1$. \\
		\item [\rm (ii)] If $1<p_1,\ldots,p_m<\infty$, then there exists a  $r>1$ such that $\vec{\omega}\in A_{\vec{p}/r}^{\infty}(\varphi)$, where $ p_j/r>1 $, $ 
		j=1,\dots ,m $.
	\end{itemize}
		\end {lemma}
\subsection{Sharp Maximal Pointwise Estimates}
\quad\quad   In this subsection, the pointwise estimates for the sharp maximal function of the multilinear commutator is established. The following is the main result of this subsection.

\begin{theorem}\label {Theorem3.1} 
	Suppose that $g_{\lambda}^{*}$ is the new multilinear Littlewood--Paley function with generalized kernel as in Definition \ref{definition1.1} and $\vec{b}\in BMO_{\vec{\theta}}^{m}(\varphi)$, $\vec{\theta}=(\theta_{1},\ldots,\theta_{m})$ with $\theta_{j}\geq 0$, $j=1,\ldots,m$. Let $\sum_{k=1}^{\infty}kC_{k}<\infty$, $0<\delta<\varepsilon<1/m$, $q'<l<\infty$ and $\eta>(\max_{1\leq j \leq m}\theta_{j})/ (1/ \delta-1/\varepsilon)$. There exists a constant $C>0$ such that,
	\begin{align*} 
		M_{\delta,\varphi,\eta}^{\sharp,\triangle}(g_{\lambda,\sum \vec b}^{*}(\vec f))(x)\leq C \sum_{j=1}^{m}\| b_{j} \|_{BMO_{\theta_{j}}(\varphi)}
		\left(M_{\varepsilon,\varphi,\eta}^\triangle (g_{\lambda}^{*}(\vec f))(x)+\mathcal{M}_{l,\varphi,\eta}(\vec f)(x)\right)
	\end{align*}
	for all $\vec f=(f_1,\ldots,f_m)$ of bounded measurable functions with compact suppport.
	\begin{proof}We only prove the case $\theta_{1}=\cdots=\theta_{m}=\theta$ for simplicity. Fix a point $x\in \mathbb{R}^n$ and for any dyadic cube $Q:=Q(x_0,r)\ni x$, we consider two cases about the sidelength $r$ : $r<1$ and $r\geq1$.
		
		\noindent{\it \rm \textbf{Case 1}:} $r<1$. 
		Let $Q^*:=14n\sqrt{mn}Q$, we split each $f_j$ as 
		$f_j=f_j^0+f_j^\infty=f_j \chi_{Q^*}+f_j \chi_{(Q^*)^{c}}$. Then
		\begin {align*}
		\prod_{j=1}^mf_j(y_j)
		&=\prod_{j=1}^m f_j^{0}(y_j)+\sum_{(\alpha_1,\ldots,\alpha_m)\in 
			\mathscr{L}}f_1^{\alpha_1}(y_1)\cdots f_m^{\alpha_m}(y_m),
		\end {align*}
		where $\mathscr{L}:=\{(\alpha_1,\ldots,\alpha_m)$: there is at least one $\alpha_j=\infty\}$. 
		
		For $\mu\in Q$ and $\lambda_j:=(b_j)_{Q*}$, $j=1,\ldots,m$. Take $\mu_{0}\in 4Q\setminus 3Q$ and
		\begin{align*}
		&A_{j}:=\left(\iint_{\mathbb{R}_{+}^{n+1}}\left(\frac{t}{|z|+t}\right)^{n\lambda}\left|\int_{({\mathbb{R}^n})^m} \sum_{(\alpha_1,\ldots,\alpha_m)\in\mathscr{L}}K_{t}(\mu _{0}-z,\vec{y})(\lambda_j-b_{j}(y_{j}))\prod_{i=1}^{m}f_{i}^{\alpha_{i}}(y_{i})d\vec{y}\right| ^{2} \frac{dzdt}{t^{n+1}}\right)^{\frac{1}{2}}.
		\end {align*} 
		We then have
		\begin{align*}
		&\left|g_{\lambda,b_{j}}^{*,j}(\vec{f})(\mu)-A_{j}\right|\\
		&\leq \Bigg(\iint_{\mathbb{R}_{+}^{n+1}}\left(\frac{t}{|z|+t}\right)^{n\lambda}\Bigg|\int_{({\mathbb{R}^n})^m}\Big(K_{t}(\mu-z,\vec{y})(b_{j}(\mu)-b_{j}(y_{j}))\prod_{i=1}^{m}f_{i}(y_{i})\\
		&\quad -\sum_{(\alpha_1,\ldots,\alpha_m)\in\mathscr{L}}K_{t}(\mu _{0}-z,\vec{y})(\lambda_j-b_{j}(y_{j}))\prod_{i=1}^{m}f_{i}^{\alpha_{i}}(y_{i})\Big)d\vec{y}\Bigg|^{2}\frac{dzdt}{t^{n+1}}\Bigg)^{\frac{1}{2}}\\
		&\leq \Bigg(\iint_{\mathbb{R}_{+}^{n+1}}\left(\frac{t}{|z|+t}\right)^{n\lambda}\Bigg|\int_{({\mathbb{R}^n})^m}K_{t}(\mu-z,\vec{y})(b_{j}(\mu)-\lambda_{j})\prod_{i=1}^{m}f_{i}(y_{i})d\vec{y}\Bigg|^{2}\frac{dzdt}{t^{n+1}}\Bigg)^{\frac{1}{2}}\\
		&\quad +\Bigg(\iint_{\mathbb{R}_{+}^{n+1}}\left(\frac{t}{|z|+t}\right)^{n\lambda}\Bigg|\int_{({\mathbb{R}^n})^m}K_{t}(\mu-z,\vec{y})(\lambda_{j}-b_{j}(y_{j}))\prod_{i=1}^{m}f_{i}^{0}(y_{i})d\vec{y}\Bigg|^{2}\frac{dzdt}{t^{n+1}}\Bigg)^{\frac{1}{2}}\\
		&\quad +\sum_{(\alpha_1,\ldots,\alpha_m)\in\mathscr{L}}\Bigg(\iint_{\mathbb{R}_{+}^{n+1}}\left(\frac{t}{|z|+t}\right)^{n\lambda}\Bigg|\int_{({\mathbb{R}^n})^m}\Big(K_{t}(\mu-z,\vec{y})-K_{t}(\mu_{0}-z,\vec{y})\Big)\Big(\lambda_{j}-b_{j}(y_{j})\Big)\\
		&\quad\times\prod_{i=1}^{m}f_{i}^{\alpha_{i}}(y_{i})d\vec{y}\Bigg|^{2}\frac{dzdt}{t^{n+1}}\Bigg)^{\frac{1}{2}}\\
		&=\left|b_{j}(\mu)-\lambda_{j}\right|g_{\lambda}^{*}(\vec{f})(\mu)+g_{\lambda}^{*}\left(f_{1}^{0},\ldots,f_{j-1}^{0},(\lambda_{j}-b_{j})f_{j}^{0},f_{j+1}^{0},\ldots,f_{m}^{0} \right)(\mu)\\
		&\quad +\sum_{(\alpha_1,\ldots,\alpha_m)\in\mathscr{L}}\Bigg(\iint_{\mathbb{R}_{+}^{n+1}}\left(\frac{t}{|z|+t}\right)^{n\lambda}\Bigg|\int_{({\mathbb{R}^n})^m}\Big(K_{t}(\mu-z,\vec{y})-K_{t}(\mu_{0}-z,\vec{y})\Big)\Big(\lambda_{j}-b_{j}(y_{j})\Big)\\
		&\quad\times\prod_{i=1}^{m}f_{i}^{\alpha_{i}}(y_{i})d\vec{y}\Bigg|^{2}\frac{dzdt}{t^{n+1}}\Bigg)^{\frac{1}{2}}.
		\end {align*} 
		So
		\begin{align*}
		&\left(\frac1{|Q|}\int_Q \bigg|g_{\lambda, b_{j}}^{*,j}(\vec{f})(\mu)-A_{j} \bigg|^\delta d\mu\right)^{\frac1\delta}\\
	&\leq C\left(\frac1{|Q|}\int_Q 
		|\left(b_{j}(\mu)-\lambda_{j}\right)g_{\lambda}^{*}(\vec{f})(\mu)|^\delta
		d\mu\right)^{\frac1\delta}\\
	&  \quad +C\left(\frac1{|Q|}\int_Q 
	|g_{\lambda}^{*}\left(f_{1}^{0},\ldots,f_{j-1}^{0},(\lambda_{j}-b_{j})f_{j}^{0},f_{j+1}^{0},\ldots,f_{m}^{0} \right)(\mu)|^\delta d\mu\right)^{\frac1\delta}\\
	&  \quad + C\sum_{(\alpha_1,\ldots,\alpha_m)\in\mathscr{L}}\Bigg(\frac1{|Q|}\int_Q \bigg(\iint_{\mathbb{R}_{+}^{n+1}}\left(\frac{t}{|z|+t}\right)^{n\lambda}\bigg|\int_{({\mathbb{R}^n})^m}\big(K_{t}(\mu-z,\vec{y})-K_{t}(\mu_{0}-z,\vec{y})\big)\\
	&\quad\times\Big(\lambda_{j}-b_{j}(y_{j})\Big)\prod_{i=1}^{m}f_{i}^{\alpha_{i}}(y_{i})d\vec{y}\bigg|^{2}\frac{dzdt}{t^{n+1}}\bigg)^{\frac{\delta}{2}}d\mu\Bigg)^{\frac{1}{\delta}}\\
	&:=H_{1}+H_{2}+H_{3}\\
	&:=H_{1}+H_{2}+\sum_{(\alpha_1,\ldots,\alpha_m)}H_{3 \alpha_1,\ldots,\alpha_m}.
	\end {align*}
					
	Choosing $1<q_{2}<\min\left\{\varepsilon /\delta, 1 / (1-\delta)\right\}$, then select $q_{1}$ such that  $1/q_{1}+1/q_{2}=1$, so we have $1<q_{1},q_{2}<\infty$ and $\delta q_{1}>1$. By H\"{o}lder's inequality and Lemma \ref{Lemma3.1}, we obtain that
	\begin{align*}
	H_{1}&\leq C\left( \frac1{|Q|}\int_Q 
	|b_j(\mu)-\lambda_j|^{\delta q_1}d\mu\right)^{\frac1{\delta q_1}}
	\left( \frac1{|Q|}\int_Q |g_{\lambda}^{*}(\vec{f})(\mu)|^{\delta q_2}d\mu\right)^{\frac1{\delta q_2}}\\
	&\leq C\|b_j\|_{BMO_\theta(\varphi)} \left( 
	\frac1{\varphi(Q)^\eta|Q|}\int_Q |g_{\lambda}^{*}(\vec{f})(\mu)|^\varepsilon 
	d\mu\right)^{\frac{1}{\varepsilon} }\\
	&\leq C\|b_j\|_{BMO_\theta(\varphi)} M_{\varepsilon,\varphi,\eta}^\triangle 
	(g_{\lambda}^{*}(\vec{f}))(x).
	\end {align*}
						
     Let $1/h_j+1/h'_j=1$ and $h_j=l/s_j$. For $s_j \leq q'<l$, $j=1,\ldots,m$, so $h_j>1$. Using Lemma \ref{Lemma2.7}, $g_{\lambda}^{*}:L^{s_1}\times\cdots\times L^{s_m}\to L^{s,\infty}$, H\"{o}lder's inqualitiy and Lemma \ref{Lemma3.1}, we obtain
		\begin{align*}
		H_{2} &\leq C|Q|^{-\frac{1}{s}}\| 
		g_{\lambda}^{*}\left(f_{1}^{0},\ldots,f_{j-1}^{0},(\lambda_{j}-b_{j})f_{j}^{0},f_{j+1}^{0},\ldots,f_{m}^{0}\right) \|_{L^{s,\infty}}\\
		&\leq C\prod_{i=1,i\neq j}^{m} 
		\left(\frac1{|Q^{*}|}\int_{Q^{*}} |f_i(y_i)|^{s_i}dy_i \right)^{\frac1{s_i}} \left(\frac1{|Q^{*}|}\int_{Q^{*}} |(\lambda_j-b_j(y_{j}))f_j(y_j)|^{s_j}dy_j \right)^{\frac1{s_j}}  \\
		&\leq C  \prod_{i=1}^{m} \left(\frac1{|Q^{*}|}\int_{Q^{*}} |f_i(y_i)|^{l}dy_i \right)^{\frac1{l}} \left(\frac1{|Q^{*}|}\int_{Q^{*}} |\lambda_j-b_j(y_{j})|^{s_jh_{j}'}dy_j \right)^{\frac1{s_jh_{j}'}}\\
		&\leq C \| b_j\|_{BMO_\theta(\varphi)}\prod_{i=1}^{m}\left(\frac1{\varphi(Q^{*})^\eta|Q^{*}|}\int_{Q^{*}} |f_i(y_i)|^{l}dy_i \right)^{\frac1l}   \\
		&\leq C \| b_j\|_{BMO_\theta(\varphi)}\mathcal{M}_{l,\varphi,\eta}(\vec f)(x).
		\end {align*}
							
		Let $\Delta_k:=Q(\mu_{0}, 2^{k}\sqrt{mn}|\mu-\mu_{0}|)$, $ \mu\in Q $, $k\in \mathbb{N}_+$, $h=l/q'$ and $1/h+1/h'=1$. For $1< q'<l$, then it can be inferred that $h>1$. Since $\mu\in Q$, $\mu_{0}\in 4Q\setminus 3Q$ and $\Delta_2 
		\subset Q^*$, we conclude $(\mathbb{R}^n)^m\setminus (Q^*)^m\subset 
		(\mathbb{R}^n)^m\setminus (\Delta_2)^m$, $|\mu-\mu_{0}|\sim r$ and $\Delta_{k+2} \subset 2^kQ^*$. Taking $N>m\eta/l+\theta$, it follows from
		H\"{o}lder's inquality, Minkowski's inequality, the smoothness condition \eqref{1.3} and Lemma \ref{Lemma3.1} that
		\begin{align*}
		H_{3 \alpha_1, \ldots, \alpha_m}
		&\leq\frac{1}{|Q|}\int_{ Q}\Bigg(\iint_{\mathbb{R}_{+}^{n+1}}\left(\frac{t}{|z|+t}\right)^{n\lambda}\Big(\int_{{({\mathbb{R}^n})^m}\backslash\left(Q^{*}\right)^{m}}\big|K_{t}(\mu-z,\vec{y})-K_{t}(\mu_{0}-z,\vec{y})\big|\\
		& \quad \times |\lambda_{j}-b_{j}(y_{j})|\prod_{i=1}^{m}|f_{i}(y_{i})| d\vec{y} \Big)^{2}  \frac{dzdt}{t^{n+1}}\Bigg)^{\frac{1}{2}}d\mu\\
		&\leq\frac{1}{|Q|}\int_{ Q}\Bigg(\iint_{\mathbb{R}_{+}^{n+1}}\left(\frac{t}{|z|+t}\right)^{n\lambda}\Big(\int_{{({\mathbb{R}^n})^m}\backslash\left(\Delta_{2}\right)^{m}}\big|K_{t}(\mu-z,\vec{y})-K_{t}(\mu_{0}-z,\vec{y})\big|\\
		& \quad \times |\lambda_{j}-b_{j}(y_{j})|\prod_{i=1}^{m}|f_{i}(y_{i})| d\vec{y} \Big)^{2}  \frac{dzdt}{t^{n+1}}\Bigg)^{\frac{1}{2}}d\mu\\
		& \leq\frac{1}{|Q|}\int_{ Q}\sum_{k=1}^{\infty}\Bigg( \iint_{\mathbb{R}_{+}^{n+1}}\left(\frac{t}{|z|+t}\right)^{n\lambda}\bigg(\int_{\left(\Delta_{k+2}\right)^{m}\backslash\left(\Delta_{k+1}\right)^{m}}\big|K_{t}(\mu-z,\vec{y})\\
		& \quad 	-K_{t}(\mu_{0}-z,\vec{y})\big|^{q}d\vec{y}\bigg)^{\frac{2}{q}}     \bigg(\int_{\left(\Delta_{k+2}\right)^{m}}\big(|\lambda_{j}-b_{j}(y_{j})|\prod_{i=1}^{m}|f_{i}(y_{i})|\big)^{q'}d\vec{y}\bigg)^{\frac{2}{q'}}\frac{dzdt}{t^{n+1}}    \Bigg)^{\frac{1}{2}}d\mu\\
		& \leq\frac{C}{|Q|}\int_{ Q}\sum_{k=1}^{\infty}C_{k}2^{-\frac{kmn}{q'}}|\mu-\mu_{0}|^{-\frac{mn}{q'}}\left(1+2^{k}|\mu-\mu_{0}|\right)^{-N}\\
		& \quad \times \prod_{i=1,i\neq j}^{m}\left(\int_{2^kQ^*}|f_{i}(y_{i})|^{q'}dy_{i}\right)^{\frac{1}{q'}}\left(\int_{2^kQ^*}\left(|\lambda_{j}-b_{j}(y_{j})||f_{j}(y_{j})|\right)^{q'}dy_{j}\right)^{\frac{1}{q'}}d\mu\\
		&\leq \frac{C}{|Q|}\int_Q \sum_{k=1}^{\infty} \frac{C_k |2^kQ^*|^{\frac{m}{q'}}}{(2^kr)^{\frac{mn}{q'}}(1+2^kr)^{N}}\prod_{i=1,i\neq j}^{m}\left(\frac{1}{|2^kQ^*|}\int_{2^kQ^*}|f_{i}(y_{i})|^{q'}dy_{i}\right)^{\frac{1}{q'}}\\
		& \quad \times \left(\frac{1}{|2^kQ^*|}\int_{2^kQ^*}\left(|\lambda_{j}-b_{j}(y_{j})||f_{j}(y_{j})|\right)^{q'}dy_{j}\right)^{\frac{1}{q'}}d\mu\\
		&\leq \frac{C}{|Q|}\int_Q \sum_{k=1}^{\infty} \frac{C_k |2^kQ^*|^{\frac{m}{q'}} \varphi(2^kQ^*)^{\frac{m\eta}{l}}}{(2^kr)^{\frac{mn}{q'}}(1+2^kr)^{N}}
		\prod_{i=1}^{m}\left(\frac1{\varphi(2^kQ^*)^{\eta}|2^kQ^*|}	\int_{2^kQ^*}|f_i(y_i)|^{l} dy_i\right)^{\frac1{l}}\\
		& \quad \times \left( \frac1{|2^kQ^*|} \int_{2^kQ^*}|\lambda_j-b_j(y_{j})|^{q'h'} 
		dy_2\right)^{\frac1{q'h'}}d\mu \\
		&\leq \frac{C}{|Q|}\int_Q\sum_{k=1}^{\infty} \frac{C_k 
		|2^kQ^*|^{\frac{m}{q'}} \varphi(2^kQ^*)^{\frac{m\eta}{l}}}{(2^kr)^{\frac{mn}{q'}}(1+2^kr)^{N}} \mathcal{M}_{l,\varphi,\eta}(\vec f)(x) k\| b_j \|_{BMO_\theta (\varphi)} \varphi(2^kQ^*)^\theta  d\mu\\
		&\leq C\| b_j \|_{BMO_\theta (\varphi)} \mathcal{M}_{l,\varphi,\eta}(\vec f)(x).
		\end {align*}
		Thus, $ H_{3 } \leq C \| b_j \|_{BMO_\theta (\varphi)} \mathcal{M}_{l,\varphi,\eta}(\vec f)(x)$.
		
		\noindent{\it \rm \textbf{Case 2}:} $r\geq1$. 
		Let ${\widetilde Q}^*=8Q$. Splitting each $f_j$ as $f_j=f_j^0+f_j^\infty=f_j \chi_{{\widetilde Q}^*}+f_j \chi_{({\widetilde Q}^*)^{c}}$. Then
		\begin {align*}
		\prod_{j=1}^mf_j(y_j)
		&=\prod_{j=1}^m f_j^{0}(y_j)+\sum_{(\alpha_1,\ldots,\alpha_m)\in 
		\mathscr{L}}f_1^{\alpha_1}(y_1),\ldots,f_m^{\alpha_m}(y_m),
		\end {align*}
		where  $\mathscr{L}:=\{(\alpha_1,\ldots,\alpha_m)$: there is at least one $\alpha_j=\infty\}$. 
								
		For $\mu\in Q$ and $\widetilde\lambda_j:=(b_j)_{{\widetilde Q}^*}$, $j=1,2,\ldots,m$. We conclude that
		\begin{align*}
		g_{\lambda,b_{j}}^{*,j}(\vec{f})(\mu)
		&=\Bigg(\iint_{\mathbb{R}_{+}^{n+1}}\left(\frac{t}{|z|+t}\right)^{n\lambda}\bigg|\int_{({\mathbb{R}^n})^m}K_{t}(\mu-z,\vec{y})(b_{j}(\mu)-\widetilde\lambda_{j})\prod_{i=1}^{m}f_{i}(y_{i})\\
		&\quad +K_{t}(\mu-z,\vec{y})(\widetilde\lambda_{j}-b_{j}(y_{j}))\prod_{i=1}^{m}f_{i}(y_{i})d\vec{y}\bigg|^{2}\frac{dzdt}{t^{n+1}}\Bigg)^{\frac{1}{2}}\\
		&\leq \Bigg(\iint_{\mathbb{R}_{+}^{n+1}}\left(\frac{t}{|z|+t}\right)^{n\lambda}\bigg|\int_{({\mathbb{R}^n})^m}K_{t}(\mu-z,\vec{y})(b_{j}(\mu)-\widetilde\lambda_{j})\prod_{i=1}^{m}f_{i}(y_{i})d\vec{y}\bigg|^{2}\frac{dzdt}{t^{n+1}}\Bigg)^{\frac{1}{2}}\\
		&\quad +\Bigg(\iint_{\mathbb{R}_{+}^{n+1}}\left(\frac{t}{|z|+t}\right)^{n\lambda}\bigg|\int_{({\mathbb{R}^n})^m}K_{t}(\mu-z,\vec{y})(\widetilde\lambda_{j}-b_{j}(y_{j}))\prod_{i=1}^{m}f_{i}^{0}(y_{i})d\vec{y}\bigg|^{2}\frac{dzdt}{t^{n+1}}\Bigg)^{\frac{1}{2}}\\
		&\quad +\sum_{(\alpha_1,\ldots,\alpha_m)}\Bigg(\iint_{\mathbb{R}_{+}^{n+1}}\left(\frac{t}{|z|+t}\right)^{n\lambda}\bigg|\int_{({\mathbb{R}^n})^m}K_{t}(\mu-z,\vec{y})(\widetilde\lambda_{j}-b_{j}(y_{j}))\\
		&\quad \times \prod_{i=1}^{m}f_{i}^{\alpha_{i}}(y_{i})d\vec{y}\bigg|^{2}\frac{dzdt}{t^{n+1}}\Bigg)^{\frac{1}{2}}
		\end{align*}
		So
		\begin{align*}
		&\left(\frac1{\varphi(Q)^{\eta}|Q|}\int_Q \big|g_{\lambda, b_{j}}^{*, j}(\vec{f})(\mu)\big|^\delta d\mu\right)^{\frac1\delta}\\
		&\leq \frac C{\varphi(Q)^{\eta/\delta}}\left(\frac1{|Q|}\int_Q 
		\big|(b_j(\mu)-\widetilde\lambda_j)g_{\lambda}^{*}(\vec{f})(\mu)\big|^\delta 
		d\mu\right)^{\frac1\delta}\\
		& \quad +\frac C{\varphi(Q)^{\eta/\delta}}\left(\frac1{|Q|}\int_Q 
		\big|g_{\lambda}^{*}\left(f_{1}^{0},\ldots,f_{j-1}^{0},(\widetilde\lambda_{j}-b_{j})f_{j}^{0},f_{j+1}^{0},\ldots,f_{m}^{0}\right)(\mu)\big|^\delta d\mu\right)^{\frac1\delta}\\
		& \quad +C\sum_{(\alpha_1,\ldots,\alpha_m)\in\mathscr{L}}\frac1{\varphi(Q)^{\eta/\delta}}\left(\frac1{|Q|}\int_Q \big|g_{\lambda}^{*}\left(f_{1}^{\alpha_{1}},\ldots,f_{j-1}^{\alpha_{j-1}},(\widetilde\lambda_{j}-b_{j})f_{j}^{\alpha_j},f_{j+1}^{\alpha_{j+1}},\ldots,f_{m}^{\alpha_m}\right)(\mu)\big|^\delta d\mu\right)^{\frac1\delta}\\
		&:=\widetilde H_{1}+\widetilde H_{2}+\widetilde H_{3}\\
		&:=\widetilde H_{1}+\widetilde H_{2}+\sum_{(\alpha_1,\ldots,\alpha_m)\in\mathscr{L}}\widetilde H_{3 \alpha_1,\ldots,\alpha_m}.
		\end {align*}
									
		Since $\eta>\theta/ (1/ \delta-1/\varepsilon)$ and by H\"{o}lder's inequality,  similar to the estimate of $H_{1}$, we obtain that
		\begin{align*}
		\widetilde H_{1}\leq C\|b_j\|_{BMO_\theta(\varphi)} 
		M_{\varepsilon,\varphi,\eta}^\triangle (g_{\lambda}^{*}(\vec{f}))(x).
		\end {align*}
										
	Now we estimate $\widetilde H_{2}$. Similar to the estimate of $H_{2}$, since $\eta>\theta/ (1/ \delta-1/\varepsilon)>\theta/(1/\delta -m/l) $, we have
	\begin{align*}
		\widetilde H_{2}
		&\leq C \| b_j\|_{BMO_\theta(\varphi)}\mathcal{M}_{l,\varphi,\eta}(\vec f)(x).
		\end {align*}
											
	It is easy to get that $ {\textstyle \sum_{j=1}^{m}} |\mu-y_j|\sim 2^{k}r$ for $\mu\in Q$ and $(y_1,y_2,\ldots,y_{m})\in (2^{k+3}Q)^m \setminus (2^{k+2}Q)^m$. Take $N >m\eta/l+\theta$, by H\"{o}lder's inqualitiy, Minkowski's inequality and the size condition \eqref{1.1}, so
	\begin {align*}	
	\widetilde H_{3 \alpha_1,\ldots,\alpha_m}
	&\leq\frac C{\varphi(Q)^{\eta/\delta}{|Q|}}\int_{ Q}\Bigg(\iint_{\mathbb{R}_{+}^{n+1}}\left(\frac{t}{|\mu-z|+t}\right)^{n\lambda}\Big(\int_{{({\mathbb{R}^n})^m}\backslash\left(\widetilde Q^{*}\right)^{m}}\big|K_{t}(z,\vec{y})\big|\big|\widetilde\lambda_{j}-b_{j}(y_{j})\big|\\
	&\quad \times\prod_{i=1}^{m}\big|f_{i}(y_{i})\big|d\vec{y} \Big)^{2} \frac{dzdt}{t^{n+1}}\Bigg)^{\frac{1}{2}}d\mu\\
	&\leq\frac C{\varphi(Q)^{\eta/\delta}{|Q|}}\int_{ Q}\int_{{({\mathbb{R}^n})^m}\backslash\left(\widetilde Q^{*}\right)^{m}}\Bigg(\iint_{\mathbb{R}_{+}^{n+1}}\left(\frac{t}{|\mu-z|+t}\right)^{n\lambda}\big|K_{t}(z,\vec{y})\big|^{2}\frac{dzdt}{t^{n+1}}\Bigg)^{\frac{1}{2}}\\
	&\quad \times\big|\widetilde\lambda_{j}-b_{j}(y_{j})\big|\prod_{i=1}^{m}\big|f_{i}(y_{i})\big|d\vec{y} d\mu\\  
	&\leq \frac C{\varphi(Q)^{\eta/\delta}{|Q|}}\int_{ Q}\sum_{k=1}^{\infty}\int_{\left(2^{k+3}Q\right)^{m}\backslash\left(2^{k+2}Q\right)^{m}}\frac{\big|\widetilde\lambda_{j}-b_{j}(y_{j})\big|\prod_{i=1}^{m}\big|f_{i}(y_{i})\big|}{\left(\sum_{j=1}^{m}|\mu-y_{j}|\right)^{mn}(1+\sum_{j=1}^{m}|\mu-y_{j}|)^{N}}d\vec{y}d\mu\\
	&\leq \frac C{\varphi(Q)^{\eta/\delta}{|Q|}}\int_{ Q}\sum_{k=1}^{\infty}\frac{|2^{k+3}Q|^{m}}{(2^{k}r)^{mn}(1+2^{k}r)^{N}}\prod_{i=1}^{m}\left(\frac{1}{|2^{k+3}Q|}\int_{2^{k+3}Q }|f_{i}(y_{i})|^{l}dy_{i}\right)^{\frac{1}{l}}\\	
	&\quad \times\left(\frac{1}{|2^{k+3}Q|}\int_{2^{k+3}Q}\big|\widetilde\lambda_{j}-b_{j}(y_{j})\big|^{l'}dy_{j}\right)^{\frac{1}{l'}}d\mu\\ 
	&\leq \frac C{\varphi(Q)^{\eta/\delta}{|Q|}}\sum_{k=1}^{\infty}\frac{|2^{k+3}Q|^{m}\varphi(2^{k+3}Q)^{\frac{m\eta}{l}}}{(2^{k}r)^{mn}(1+2^{k}r)^{N}}\prod_{i=1}^{m}\left(\frac{1}{\varphi(2^{k+3}Q)^{\eta}|2^{k+3}Q|}\int_{2^{k+3}Q }|f_{i}(y_{i})|^{l}dy_{i}\right)^{\frac{1}{l}}\\	
	&\quad \times\left(\frac{1}{|2^{k+3}Q|}\int_{2^{k+3}Q}\big|\widetilde\lambda_{j}-b_{j}(y_{j})\big|^{l'}dy_{j}\right)^{\frac{1}{l'}}\\ 
	&\leq \frac C{\varphi(Q)^{\eta/\delta}{|Q|}}\sum_{k=1}^{\infty}\frac{|2^{k+3}Q|^{m}\varphi(2^{k+3}Q)^{\frac{m\eta}{l}}}{(2^{k}r)^{mn}(1+2^{k}r)^{N}}\mathcal{M}_{l,\varphi,\eta}(\vec f)(x)k\| b_j \|_{BMO_\theta (\varphi)}\varphi(2^{k+3}Q)^{\theta}\\
	&\leq C\sum_{k=1}^{\infty}k\left(1+2^{k}r\right)^{\frac{m\eta}{l}-N+\theta}\| b_j \|_{BMO_\theta (\varphi)}\mathcal{M}_{l,\varphi,\eta}(\vec f)(x)\\
	&\leq C\sum_{k=1}^{\infty}k\left(2^{k}\right)^{\frac{m\eta}{l}-N+\theta}\| b_j \|_{BMO_\theta (\varphi)}\mathcal{M}_{l,\varphi,\eta}(\vec f)(x)\\
	&\leq C\| b_j \|_{BMO_\theta (\varphi)}\mathcal{M}_{l,\varphi,\eta}(\vec f)(x).
	\end {align*}
	
	In short, if $r<1$, we obtain
	\begin{align*}
	&\left(\frac1{|Q|}\int_Q \left|g_{\lambda, b_{j}}^{*,j}(\vec{f})(\mu)-A_{j}\right|^\delta  d\mu\right)^{\frac1\delta} \leq C \| b_{j} \|_{BMO_{\theta}(\varphi)}\left(M_{\varepsilon,\varphi,\eta}^\triangle \big(g_{\lambda}^{*}(\vec f)\big)(x)+\mathcal{M}_{l,\varphi,\eta}(\vec f)(x)\right),
	\end{align*}
	then
	\begin{align*}
	&\inf_{c}\left(\frac1{|Q|}\int_Q \Big|\big|g_{\lambda, \sum 
	\vec{b}}^{*}(\vec{f})(\mu)\big|^\delta-c \Big| d\mu\right)^{\frac1\delta}
	\leq \left(\frac1{|Q|}\int_Q \Big|\big|g_{\lambda, \sum \vec{b}}^{*}(\vec{f})(\mu)\big|^\delta-\big|\sum_{j=1}^{m}A_{j}\big|^{\delta} \Big| d\mu\right)^{\frac1\delta}\\
	&\leq C\left(\frac1{|Q|}\int_Q \bigg|\sum_{j=1}^{m}g_{\lambda, b_{j}}^{*,j}(\vec{f})(\mu)-\sum_{j=1}^{m}A_{j}\bigg|^\delta  d\mu\right)^{\frac1\delta}
	\leq C\sum_{j=1}^{m}\left(\frac1{|Q|}\int_Q \left|g_{\lambda, b_{j}}^{*,j}(\vec{f})(\mu)-A_{j}\right|^\delta  d\mu\right)^{\frac1\delta}\\
	&\leq  C \sum_{j=1}^{m}\| b_{j} \|_{BMO_{\theta}(\varphi)}	\left(M_{\varepsilon,\varphi,\eta}^\triangle \big(g_{\lambda}^{*}(\vec f)\big)(x)+\mathcal{M}_{l,\varphi,\eta}(\vec f)(x)\right).
	\end{align*}

	And if $r \geq 1$, we have
	\begin{align*}
	&\left(\frac1{\varphi(Q)^{\eta}|Q|}\int_Q \left|g_{\lambda, b_{j}}^{*,j}(\vec{f})(\mu)\right|^\delta d\mu\right)^{\frac1\delta} \leq C \| b_{j} \|_{BMO_{\theta}(\varphi)}\left(M_{\varepsilon,\varphi,\eta}^\triangle (g_{\lambda}^{*}(\vec f))(x)+\mathcal{M}_{l,\varphi,\eta}(\vec f)(x)\right),
	\end{align*}
	then
	\begin{align*}
	&\left(\frac1{\varphi(Q)^{\eta}|Q|}\int_Q \left|g_{\lambda, \sum \vec{b}}^{*}(\vec{f})(\mu)\right|^\delta d\mu\right)^{\frac1\delta}\\
	&\leq C\sum_{j=1}^{m}\left(\frac1{\varphi(Q)^{\eta}|Q|}\int_Q \left|g_{\lambda, b_{j}}^{*,j}(\vec{f})(\mu)\right|^\delta d\mu\right)^{\frac1\delta}\\
	&\leq  C \sum_{j=1}^{m}\| b_{j} \|_{BMO_{\theta}(\varphi)}
	\left(M_{\varepsilon,\varphi,\eta}^\triangle \big(g_{\lambda}^{*}(\vec f)\big)(x)+\mathcal{M}_{l,\varphi,\eta}(\vec f)(x)\right).
	\end{align*}
											
	According to the above two cases, we obtain
	\begin{align*} 
	&M_{\delta,\varphi,\eta}^{\sharp,\triangle}\big(g_{\lambda,\sum \vec b}^{*}(\vec f)\big)(x)\simeq \Bigg(\mathop{\sup}\limits_{{\rm Q}\ni x,r<1}\mathop{\inf}\limits_{ c}\frac{1}{|Q|}\int_{Q}\left|\big|g_{\lambda,\sum \vec b}^{*}(\vec f)(\mu)\big|^{\delta}-c\right|d\mu\\
	&\ \ \ \ \ \ \ \ \ \ \ \ \ \ \ \ \ \ \ \ \ \ \ \ \ \ \  \quad +\mathop{\sup}\limits_{{\rm Q}\ni x,r\geq1}\frac{1}{\varphi(Q)^{\eta}|Q|}\int_{Q}\left|g_{\lambda,\sum \vec b}^{*}(\vec f)(\mu)\right|^{\delta}d\mu\Bigg)^{\frac{1}{\delta}}\\	
	&\ \ \ \ \ \ \ \ \ \ \ \ \ \ \ \ \ \ \ \ \ \ \ \ \ \ \ 
	\leq C \sum_{j=1}^{m}\| b_{j} \|_{BMO_{\theta}(\varphi)}
	\left(M_{\varepsilon,\varphi,\eta}^\triangle (g_{\lambda}^{*}(\vec f))(x)+\mathcal{M}_{l,\varphi,\eta}(\vec f)(x)\right).
	\end{align*}
											
	The proof of the Theorem \ref{Theorem3.1} is finished.
											
	\end{proof}
	\end{theorem}
	\subsection{Boundedness of Multilinear Commutator on weighted Lebesgue Spaces}
\quad\quad In this subsection, with the help of Theorem \ref{Theorem3.1}, we obtain the boundedness of multilinear commutator on weighted Lebesgue spaces.
	\begin{theorem}\label{Theorem3.2} 
	Suppose that $m\geq2$, $g_{\lambda}^{*}$ is the new multilinear Littlewood--Paley function with generalized kernel as in Definition \ref {definition1.1}. Let ${\textstyle \sum_{k=1}^{\infty }}  k C_k<\infty$, $\vec{\theta}=(\theta_1,\ldots,\theta_m) $, $ \theta_j\geq 0 $, $ j=1,\ldots,m $, $\vec{p}=(p_1,\ldots,p_m)$, $1/p=1/p_1+\cdots+1/p_m$, $\vec \omega=(\omega_1,\ldots,\omega_m)\in A_{\vec{p}/q'}^\infty(\varphi)$, $v_{\vec\omega}= {\textstyle \prod_{j=1}^{m}} \omega_j^{p/p_j}$ and $\vec{b}\in BMO_{\vec\theta}^m(\varphi)$. If $q'<p_j<\infty$, $j=1,\ldots,m$, then there exists a constant $C>0$ such that
	\begin{align*}
	\| g_{\lambda,\sum \vec b}^{*}(\vec f)\|_{L^{p}(v_{\vec\omega})}\le C \sum_{j=1}^{m}\|b_{j} \|_{BMO_{\theta_{j}}(\varphi)}\prod_{j=1}^{m}\| f_j \|_{L^{p_j}(\omega_j)}.\nonumber
	\end{align*}								
	\end {theorem}
\begin{proof}By Lemma \ref{Lemma3.2} for $\vec \omega \in A_{\vec p/q'}^\infty(\varphi)$, there exists a $k>1$, such that $\vec \omega \in A_{\vec 	p/(q'k)}^\infty(\varphi)$ and $p_{j}/(q'k)>1$, $j=1,\ldots,m$. 
	Denote $ l=q'k$, then $\vec \omega \in A_{\vec p/l}^\infty(\varphi)$, $l>q'$ and $ 
	p_{j}/l>1 $, $ j=1,\dots ,m $.  
	Take $\eta=\eta_0$ in Lemma \ref{Lemma2.5} for $\vec \omega \in A_{\vec	p/l}^\infty(\varphi)$. It concludes from Lemma \ref{Lemma2.2} and Lemma \ref{Lemma2.1} that $v_{\vec \omega} \in A_{mp/l}^\infty (\varphi) \subset A_{p/\delta }^\infty (\varphi)$, where $\varepsilon$ and $\delta$ are taken by
	\begin{align*}
		0<\varepsilon<\frac1m, 0<\delta<\frac 1{(\frac 1 \varepsilon +\frac{\max_{1\leq j \leq m}\theta_{j}}{\eta_0})}.
	\end{align*}
	
	It follows from	Lemma \ref{Lemma2.3}, Theorem \ref{Theorem3.1} and Theorem \ref{Theorem 2.1} that						
	\begin{align*}
		\| g_{\lambda,\sum \vec b}^{*}(\vec f)\|_{L^{p}(v_{\vec\omega})}
		&=\left \| |g_{\lambda,\sum \vec b}^{*}(\vec f)|^\delta \right \| _{L^{p/ \delta}(v_{\vec\omega})}^{1/\delta }
		\leq \| M_{\varphi,\eta_0}^{\triangle}(|g_{\lambda,\sum \vec b}^{*}(\vec f)|^\delta) \|_{L^{p/ \delta}(v_{\vec\omega})}^{1/\delta }\\
		&\leq C\| M_{\delta,\varphi,\eta_0}^{\sharp,\triangle}(g_{\lambda,\sum \vec b}^{*}(\vec f)) \|_{L^{p}(v_{\vec\omega})}\\
		& \leq  C\sum_{j=1}^{m}\|b_{j} \|_{BMO_{\theta_{j}}(\varphi)}\left ( \| 
		M_{\varepsilon,\varphi,\eta_0}^\triangle (g_{\lambda}^{*}(\vec 	f))\|_{L^{p}(v_{\vec\omega})}+ \|\mathcal{M}_{l,\varphi,\eta_0}(\vec 	f)\|_{L^{p}(v_{\vec\omega})} \right ) \\
		&\leq C \sum_{j=1}^{m}\|b_{j} \|_{BMO_{\theta_{j}}(\varphi)}
		\|\mathcal{M}_{l,\varphi,\eta_0}(\vec f)\|_{L^{p}(v_{\vec\omega})},
	\end{align*}
	where
\begin{align*}
	\| M_{\varepsilon,\varphi,\eta_0}^\triangle (g_{\lambda}^{*}(\vec 
	f))\|_{L^{p}(v_{\vec\omega})}
	&\leq  C\| M_{\varepsilon, \varphi,\eta_0}^{\sharp, \triangle} (g_{\lambda}^{*}(\vec f)) 	\|_{L^{p}(v_{\vec\omega})}	\leq C \| \mathcal{M}_{q',\varphi,\eta_0}(\vec f) 	\|_{L^{p}(v_{\vec\omega})}\\
	&\leq C \| \mathcal{M}_{l,\varphi,\eta_0}(\vec f) \|_{L^{p}(v_{\vec\omega})}.
\end{align*}

Then by Lemma \ref{Lemma2.5},
\begin{align*}
	\| g_{\lambda,\sum \vec b}^{*}(\vec f)\|_{L^{p}(v_{\vec\omega})}
	&\leq C \sum_{j=1}^{m}\|b_{j} \|_{BMO_{\theta_{j}}(\varphi)}	\|\mathcal{M}_{\varphi,\eta_0}(|\vec f|)^l\|^{1/l}_{L^{p/l}(v_{\vec\omega})}\\
	&\leq C  \sum_{j=1}^{m}\|b_{j} \|_{BMO_{\theta_{j}}(\varphi)}
	\prod_{j=1}^m\| |f_j|^l \|^{1/l}_{L^{p_j/l}(\omega_j)}\\
	&= C  \sum_{j=1}^{m}\|b_{j} \|_{BMO_{\theta_{j}}(\varphi)}
	\prod_{j=1}^m\| f_j \|_{L^{p_j}(\omega_j)}.
\end{align*}
\end{proof}
\section{The Boundedness of Multilinear Iterative Commutator on weighted Lebesgue Spaces}\label{sec4}
\quad\quad In Subsection 4.1, we give some definition which will be used later. In Subsection 4.2, we establish the sharp maximal pointwise estimates. In Subsection 4.3, we prove the boundedness of multilinear iterative commutator on weighted Lebesgue spaces.
\subsection{Definitions}
\begin{definition}\label {Definition4.1}
	Let $C_{j}^{m}$ be a family of the finite subsets $\xi = \left\{\xi(1),\ldots , \xi(j)\right\}$ consisting of $j$ distinct elements from the set $\left\{1,\ldots,m\right\}$ for $j, m \in  \mathbb{N}^{+}$ and $1 \leq j \leq m$. In addition, if $a < b$, then $\xi(a) <\xi (b)$. For any $\xi \in C_{j}^{m}$, the complementary sequence of $\xi$, denoted as $\xi'$, is defined as $\left\{1, ..., m\right\} \backslash \xi$. In particular, $C_{0}^{m} =\phi$. For any m-fold sequence $\vec{b}$ and $\xi \in C_{j}^{m}$, the j-fold sequence $\vec b_{\xi} = \left\{\vec{b}_{\xi(1)}, \ldots, \vec{b}_{\xi(j)}\right\}$ is a finite subset of $\vec b = \left(b_{1}, \ldots, b_{m}\right)$.
\end{definition}
Let $\xi \in C_{j}^{m}$ and $\vec b_{\xi} = \left\{{b}_{\xi(1)}, \ldots, {b}_{\xi(j)}\right\}$, the alternative integral form of the aforementioned equation would be to state
\begin{align*}
	&g_{\lambda,\prod \vec b_{\xi}}\left(f_{1},\ldots,f_{m}\right)(x)=\Bigg(\iint_{\mathbb{R}_{+}^{n+1}}\left(\frac{t}{|z|+t}\right)^{n\lambda}\big|\int_{({\mathbb{R}^n})^m}\prod_{i=1}^{j}\left(b_{\xi(i)}(x)-b_{\xi(i)}(y_{\xi(i)})\right)\\
	&\quad \quad \quad \quad \quad \quad \quad \quad \quad \quad \quad\times K_{t}(z,y_{1},\ldots,y_{m})f_{1}(y_{1})\cdots f_{m}(y_{m})\big|^{2}\frac{dzdt}{t^{n+1}}\Bigg)^{\frac{1}{2}}.
\end{align*} 
\subsection{Sharp Maximal Pointwise Estimates}
\quad\quad   In this subsection, the pointwise estimates for the sharp maximal function of the multilinear iterated commutator is established. The following is the main result of this subsection.
\begin{theorem}\label{Theorem4.1} 
	Suppose that $g_{\lambda}^{*}$ is the new multilinear Littlewood--Paley function with generalized kernel as in Definition \ref{definition1.1}. Suppose $\vec{b}\in BMO_{\vec\theta}^m(\varphi)$, $\vec{\theta}=\left(\theta_{1},\ldots,\theta_{m}\right)$ with $\theta_{j}\ge 0$, $j=1, \ldots,m$. Let $\sum_{k=1}^{\infty}k^{m}C_{k}<\infty$, $0<\delta<\varepsilon<1/m$, $q'<l<\infty$ and $\eta >(\sum_{j=1}^{m}\theta_{j})/\left(1/\delta-1/\varepsilon\right)$. There exists a constant $C>0$ such that,
	\begin{align*} 
		&M_{\delta,\varphi,\eta}^{\sharp,\triangle}(g_{\lambda,\prod \vec b}^{*}(\vec f))(x)\leq C \prod_{j=1}^{m}\| b_j \|_{BMO_{\theta_{j}}(\varphi)} \left(M_{\varepsilon,\varphi,\eta}^\triangle (g_{\lambda}^{*}(\vec f))(x)+\mathcal{M}_{l,\varphi,\eta}(\vec f)(x)\right)\\
		&\quad \quad \quad \quad \quad \quad \quad \quad \quad
		+C\sum_{j=1}^{m-1}\sum_{\xi \in C_{j}^{m}}\prod_{i=1}^{m}\| b_{\xi (i)}\|_{BMO_{\theta_{\xi(i)}}(\varphi)}M_{\varepsilon,\varphi,\eta}^\triangle\Big(g_{\lambda,\prod \vec b_{\xi'}}^{*}(\vec f)\Big)(x)
	\end{align*}
	for all $\vec f=(f_1,\ldots,f_m)$ of bounded measurable functions with compact suppport.
\end{theorem}
\begin{proof}
	To simplify the proof, we only consider the case when $m=2$ and  $\theta_1=\theta_2=\theta$. In fact, the same approach applies to all other cases.
	
	Fix a point $x\in \mathbb{R}^n$ and for any dyadic cube $Q:=Q(x_0,r)\ni x$, we consider two cases about the sidelength $r$ : $r<1$ and $r\geq1$.
	
	\noindent{\it \rm \textbf{Case 1}:} $r<1$. Let $Q^{*}=14n\sqrt{2n}Q$. Splitting each $f_j$ as $f_j=f_j^0+f_j^\infty=f_j \chi_{Q^{*}}+f_j \chi_{(Q^{*})^{c}}$. There is
	\begin{align*}
		\prod_{j=1}^2 
		f_j(y_j)=\prod_{j=1}^2 f_j^{0}(y_j)+\sum_{(\alpha_1,\alpha_2)\in 
			\mathscr{L}}f_1^{\alpha_1}(y_1) f_2^{\alpha_2}(y_2),
		\end {align*}
		where  $\mathscr{L}:=\{(\alpha_1,\alpha_2)$: there is at least one  
		$\alpha_j=\infty\}$.
		
		For $\mu\in Q$ and $\lambda_j:=(b_j)_{Q^{*}}$, $j=1,\ldots,m$. Take $\mu_{0}\in 4Q\setminus 3Q$ and
		\begin{align*}
			&A:=\left(\iint_{\mathbb{R}_{+}^{n+1}}\left(\frac{t}{|z|+t}\right)^{n\lambda}\bigg|\int_{({\mathbb{R}^n})^2} \sum_{(\alpha_1,\alpha_2)\in\mathscr{L}}K_{t}(\mu _{0}-z,\vec{y})\prod_{i=1}^{2}(b_{i}(y_{i})-\lambda_i)f_{i}^{\alpha_{i}}(y_{i})d\vec{y}\bigg| ^{2} \frac{dzdt}{t^{n+1}}\right)^{\frac{1}{2}}.
		\end {align*} 
			We then have
	\begin{align*}
		&\left|g_{\lambda,\prod \vec b}^{*}(\vec f)(\mu)-A\right|\\
		&\leq \Bigg(\iint_{\mathbb{R}_{+}^{n+1}}\left(\frac{t}{|z|+t}\right)^{n\lambda}\Big|\int_{({\mathbb{R}^n})^2}\Big(K_{t}(\mu-z,y_{1},y_{2})\left(b_{1}(\mu)-b_{1}(y_{1})\right)\left(b_{2}(\mu)-b_{2}(y_{2})\right)f_{1}(y_{1})f_{2}(y_{2})\\
		&\quad -\sum_{(\alpha_1, \alpha_2)\in\mathscr{L}}K_{t}(\mu_{0}-z,\vec{y})\left(b_{1}(y_{1})-\lambda_{1}\right)\left(b_{2}(y_{2})-\lambda_{2}\right)f_{1}^{\alpha_1}(y_{1})f_{2}^{\alpha_2}(y_{2})\Big)dy_{1}dy_{2}\Big|^{2}\frac{dzdt}{t^{n+1}}\Bigg)^{\frac{1}{2}}\\
		&=\Bigg(\iint_{\mathbb{R}_{+}^{n+1}}\left(\frac{t}{|z|+t}\right)^{n\lambda}\Big|\int_{({\mathbb{R}^n})^2}\bigg(K_{t}(\mu-z,y_{1},y_{2})\left(b_{1}(\mu)-\lambda_1\right)\left(\lambda_2-b_{2}(\mu)\right)f_{1}(y_{1})f_{2}(y_{2})  \\
		&\quad +K_{t}(\mu-z,y_{1},y_{2})\left(b_{1}(\mu)-\lambda_1
		\right)\left(b_{2}(
		\mu)-b_2(y_{2})\right)f_{1}(y_{1})f_{2}(y_{2})  \\
		&\quad +K_{t}(\mu-z,y_{1},y_{2})\left(b_{2}(\mu)-\lambda_2\right)\left(b_{1}(\mu)-b_1(y_{1})\right)f_{1}(y_{1})f_{2}(y_{2})  \\
		&\quad + K_{t}(\mu-z,y_{1},y_{2})
		\left(b_{1}(y_{1})-\lambda_{1}\right)\left(b_{2}(y_{2})-\lambda_{2}\right)f_{1}^{0}(y_{1})f_{2}^{0}(y_{2})\\
		&\quad +\sum_{(\alpha_1,\alpha_2)\in\mathscr{L}}K_{t}(\mu-z,y_{1},y_{2})
		\left(b_{1}(y_{1})-\lambda_{1}\right)\left(b_{2}(y_{2})-\lambda_{2}\right)f_{1}^{\alpha_1}(y_{1})f_{2}^{\alpha_2}(y_{2})\\
		&\quad -\sum_{(\alpha_1,\alpha_2)\in\mathscr{L}}K_{t}(\mu_{0}-z,y_{1},y_{2})
		\left(b_{1}(y_{1})-\lambda_{1}\right)\left(b_{2}(y_{2})-\lambda_{2}\right)f_{1}^{\alpha_1}(y_{1})f_{2}^{\alpha_2}(y_{2})\bigg)dy_{1}dy_{2}\Big|^{2}\frac{dzdt}{t^{n+1}}\Bigg)^{\frac{1}{2}}\\
		&\leq \Bigg(\iint_{\mathbb{R}_{+}^{n+1}}\left(\frac{t}{|z|+t}\right)^{n\lambda} \Big| \int_{({\mathbb{R}^n})^2}K_{t}(\mu-z,y_{1},y_{2})\left(b_{1}(\mu)-\lambda_1\right)\left(\lambda_2-b_{2}(\mu)\right)\\
		&\quad
		\times f_{1}(y_{1})f_{2}(y_{2})dy_{1}dy_{2} \Big|^{2} \frac{dzdt}{t^{n+1}}\Bigg)^{\frac{1}{2}}\\
		&\quad +\Bigg(\iint_{\mathbb{R}_{+}^{n+1}}\left(\frac{t}{|z|+t}\right)^{n\lambda} \Big| \int_{({\mathbb{R}^n})^2}K_{t}(\mu-z,y_{1},y_{2})\left(b_{1}(\mu)-\lambda_1\right)\left(b_2(\mu)-b_{2}(y_{2})\right)\\
		&\quad
		\times f_{1}(y_{1})f_{2}(y_{2})dy_{1}dy_{2} \Big|^{2} \frac{dzdt}{t^{n+1}}\Bigg)^{\frac{1}{2}}\\
		&\quad +\Bigg(\iint_{\mathbb{R}_{+}^{n+1}}\left(\frac{t}{|z|+t}\right)^{n\lambda} \Big| \int_{({\mathbb{R}^n})^2}K_{t}(\mu-z,y_{1},y_{2})\left(b_{2}(\mu)-\lambda_2\right)\left(b_1(\mu)-b_{1}(y_{1})\right) \\
		&\quad
		\times f_{1}(y_{1})f_{2}(y_{2})dy_{1}dy_{2} \Big|^{2} \frac{dzdt}{t^{n+1}}\Bigg)^{\frac{1}{2}}\\
		&\quad +\Bigg(\iint_{\mathbb{R}_{+}^{n+1}}\left(\frac{t}{|z|+t}\right)^{n\lambda} \Big| \int_{({\mathbb{R}^n})^2}K_{t}(\mu-z,y_{1},y_{2})\left(b_{1}(y_{1})-\lambda_1\right)\left(b_{2}(y_{2})-\lambda_2\right)\\
		&\quad
		\times f_{1}^{0}(y_{1})f_{2}^{0}(y_{2})dy_{1}dy_{2} \Big|^{2} \frac{dzdt}{t^{n+1}}\Bigg)^{\frac{1}{2}}\\
		&\quad +\sum_{(\alpha_1,\alpha_2)\in\mathscr{L}}\Bigg(\iint_{\mathbb{R}_{+}^{n+1}}\left(\frac{t}{|z|+t}\right)^{n\lambda} \Big| \int_{({\mathbb{R}^n})^2}\bigg(K_{t}(\mu-z,y_{1},y_{2})-K_{t}(\mu_{0}-z,y_{1},y_{2})\bigg)\\
		&\quad
		\times \left(b_{1}(y_{1})-\lambda_1\right)\left(b_{2}(y_{2})-\lambda_2\right)f_{1}^{\alpha_1}(y_{1})f_{2}^{\alpha_2}(y_{2})dy_{1}dy_{2} \Big|^{2} \frac{dzdt}{t^{n+1}}\Bigg)^{\frac{1}{2}}\\
		&=\left|b_1(\mu)-\lambda_1\right|\left|b_2(\mu)-\lambda_2\right|g_{\lambda}^{*}(f_1,f_2)(\mu)+\left|b_1(\mu)-\lambda_1\right|g_{\lambda,b_{2}}^{*,2}(f_1,f_2)(\mu)\\
		&\quad +\left|b_2(\mu)-\lambda_2\right|g_{\lambda,b_{1}}^{*,1}(f_1,f_2)(\mu) +g_{\lambda}^{*}((b_1-\lambda_1)f_1^{0},(b_2-\lambda_2)f_2^{0})(\mu)\\
		&\quad +\sum_{(\alpha_1,\alpha_2)\in\mathscr{L}}\Bigg(\iint_{\mathbb{R}_{+}^{n+1}}\left(\frac{t}{|z|+t}\right)^{n\lambda} \Big| \int_{({\mathbb{R}^n})^2}\bigg(K_{t}(\mu-z,\vec{y})-K_{t}(\mu_{0}-z,\vec{y})\bigg)\\
		&\quad
		\times \left(b_{1}(y_{1})-\lambda_1\right)\left(b_{2}(y_{2})-\lambda_2\right)f_{1}^{\alpha_1}(y_{1})f_{2}^{\alpha_2}(y_{2})dy_{1}dy_{2} \Big|^{2} \frac{dzdt}{t^{n+1}}\Bigg)^{\frac{1}{2}}.
	\end{align*}

	Thus
	\begin{align*}
		&\left(\frac1{|Q|}\int_Q \Big||g_{\lambda,\prod \vec b}^{*}(\vec 
		f)(\mu)|^\delta-|A|^\delta \Big| d\mu\right)^{\frac1\delta}\\
		&\leq \left(\frac1{|Q|}\int_Q \Big|g_{\lambda,\prod \vec b}^{*}(\vec 
		f)(\mu)-A \Big|^\delta d\mu\right)^{\frac1\delta}\\
		&\leq C \left(\frac{1}{|Q|}\int_{ Q}\Big|\left(b_1(\mu)-\lambda_1\right)\left(b_2(\mu)-\lambda_2\right)g_{\lambda}^{*}(f_1,f_2)(\mu)\Big|^{\delta}d\mu\right)^{\frac{1}{\delta}}\\
		&\quad +C \left(\frac{1}{|Q|}\int_{ Q}\Big|\left(b_1(\mu)-\lambda_1\right)g_{\lambda,b_{2}}^{*,2}(f_1,f_2)(x)(\mu)\Big|^{\delta}d\mu\right)^{\frac{1}{\delta}}\\
		&\quad +C \left(\frac{1}{|Q|}\int_{ Q}\Big|\left(b_2(\mu)-\lambda_2\right)g_{\lambda,b_{1}}^{*,1}(f_1,f_2)(x)(\mu)\Big|^{\delta}d\mu\right)^{\frac{1}{\delta}}\\
		&\quad +C \left(\frac{1}{|Q|}\int_{ Q}\Big|g_{\lambda}^{*}(\left(b_1-\lambda_1\right)f_1^{0},\left(b_2-\lambda_2\right)f_2^{0})(\mu)\Big|^{\delta}d\mu\right)^{\frac{1}{\delta}}\\
		&\quad +C \sum_{(\alpha_1,\alpha_2)\in\mathscr{L}}\Bigg(\frac{1}{|Q|}\int_{Q}\bigg(\iint_{\mathbb{R}_{+}^{n+1}}\left(\frac{t}{|z|+t}\right)^{n\lambda} \Big| \int_{({\mathbb{R}^n})^2}\big(K_{t}(\mu-z,\vec{y})-K_{t}(\mu_{0}-z,\vec{y})\big)\\
		&\quad \times \left(b_{1}(y_{1})-\lambda_1\right)\left(b_{2}(y_{2})-\lambda_2\right)f_{1}^{\alpha_1}(y_{1})f_{2}^{\alpha_2}(y_{2})dy_{1}dy_{2}\Big|^{2}\frac{dzdt}{t^{n+1}}\bigg)^{\frac{\delta}{2}}d\mu\Bigg)^{\frac{1}{\delta}}\\
		&:=M_{1}+M_{2}+M_{3}+M_{4}+M_{5}\\
		&:=M_{1}+M_{2}+M_{3}+M_{4}+\sum_{(\alpha_1,\alpha_2)\in\mathscr{L}}M_{5 \alpha_1,\ldots,\alpha_m}.
		\end {align*}
		
		Since $0<\delta<\varepsilon<\infty$, choose $1<q_{3}<\min\left\{\varepsilon/
		\delta, 1/(1-\delta)\right\}$	so that $\delta q_{3}<\varepsilon$ and  $\delta q_{3}'>1$. Take $1<q_1,q_2<\infty$ such that $1/q_1+1/q_2+1/q_3=1$, so $\delta q_1>1 $ and  $\delta q_2>1 $. By H\"{o}lder's inequality, we obtain that
		\begin{align*}
			M_{1}&\leq C\left( \frac1{|Q|}\int_Q 
			|b_1(\mu)-\lambda_1|^{\delta q_1}d\mu\right)^{\frac1{\delta q_1}}
			\left( \frac1{|Q|}\int_Q |b_2(\mu)-\lambda_2|^{\delta 
				q_2}d\mu\right)^{\frac1{\delta q_2}}\\
			&  \quad \times \left( \frac1{|Q|}\int_Q |g_{\lambda}^{*}(f_1,f_2)(\mu)|^{\delta 
				q_3}d\mu\right)^{\frac1{\delta q_3}}\\
			&\leq C\|b_1\|_{BMO_\theta(\varphi)} \|b_2\|_{BMO_\theta(\varphi)}\left( 
			\frac1{\varphi(Q)^\eta|Q|}\int_Q |g_{\lambda}^{*}(f_1,f_2)(\mu)|^\varepsilon 
			d\mu\right)^{\frac{1}{\varepsilon} }\\
			&\leq C\|b_1\|_{BMO_\theta(\varphi)} 
			\|b_2\|_{BMO_\theta(\varphi)}M_{\varepsilon,\varphi,\eta}^\triangle 
			(g_{\lambda}^{*}(f_1,f_2))(x).
			\end {align*}
			
		Now we estimate $M_{2}$. Using H\"{o}lder's inequality and Lemma \ref{Lemma3.1}, we have
		\begin{align*}
		M_{2}
		&\leq C\left( \frac1{|Q|}\int_Q 
		|b_1(\mu)-\lambda_1|^{\delta q_3'}d\mu\right)^{\frac1{\delta q_3'}} \left( \frac1{|Q|}\int_Q |g_{\lambda, b_{2}}^{*,2}(f_1,f_2)(\mu)|^{\delta q_3}d\mu\right)^{\frac1{\delta q_3}}\\
		&\leq C\|b_1\|_{BMO_\theta(\varphi)} M_{\varepsilon,\varphi,\eta}^\triangle (g_{\lambda,b_{2}}^{*,2}(f_1,f_2))(x).
		\end {align*}
				
		For the term $M_{3}$, similar to the estimate of $M_{2}$, we have
		\begin{align*}
		M_{3} \leq C \|b_2\|_{BMO_\theta(\varphi)} M_{\varepsilon,\varphi,\eta}^\triangle (g_{\lambda, b_{1}}^{*,1}(f_1,f_2))(x).
		\end{align*}
				
		For the term $M_{4}$, similar to the estimates of $H_2$, we can get
		\begin{align*}
		M_{4} \leq C \| b_1\|_{BMO_\theta(\varphi)}\|b_2\|_{BMO_\theta(\varphi)}\mathcal{M}_{l,\varphi,\eta}(\vec f)(x).
		\end {align*}
		
	Let $\Delta_k:=Q(\mu_{0}, 2^{k}\sqrt{2n}|\mu-\mu_{0}|)$, $ \mu\in Q $, $k\in \mathbb{N}_+$, $h=l/q'$ and $1/h+1/h'=1$. Since $q'<l$, then it can be inferred that $h>1$. Since $\mu\in Q$, $\mu_{0}\in 4Q\setminus 3Q$ and $\Delta_2 
	\subset Q^*$, we conclude $(\mathbb{R}^n)^2\setminus (Q^*)^2\subset 
	(\mathbb{R}^n)^2\setminus (\Delta_2)^2$, $|\mu-\mu_{0}|\sim r$ and $\Delta_{k+2} \subset 2^kQ^*$. Taking $N>2\eta/l+2\theta$, it follows from  H\"{o}lder's inquality, Minkowski's inequality, the smoothness condition \eqref{1.3} and Lemma \ref{Lemma3.1} that
	\begin{align*}
		M_{5 \alpha_1,\ldots,\alpha_m}
		&\leq\frac{C}{|Q|}\int_{ Q}\Bigg(\iint_{\mathbb{R}_{+}^{n+1}}\left(\frac{t}{|z|+t}\right)^{n\lambda}\Big(\int_{{({\mathbb{R}^n})^2}\backslash\left(Q^{*}\right)^{2}}|K_{t}(\mu-z,\vec{y})-K_{t}(\mu_{0}-z,\vec{y})|\\
		& \quad \times \prod_{i=1}^{2}|b_{i}(y_{i})-\lambda_{i}||f_{i}(y_{i})| dy_{1}dy_{2} \Big)^{2}  \frac{dzdt}{t^{n+1}}\Bigg)^{\frac{1}{2}}d\mu\\
		& \leq\frac{C}{|Q|}\int_{ Q}\sum_{k=1}^{\infty}\Bigg( \iint_{\mathbb{R}_{+}^{n+1}}\left(\frac{t}{|z|+t}\right)^{n\lambda}\Bigg(\int_{\left(\Delta_{k+2}\right)^{2}\backslash\left(\Delta_{k+1}\right)^{2}}|K_{t}(\mu-z,\vec{y})-K_{t}(\mu_{0}-z,\vec{y})|\\
		& \quad 	\times \prod_{i=1}^{2}|b_{i}(y_{i})-\lambda_{i}||f_{i}(y_{i})|dy_{1}dy_{2}\Bigg)^{2}\frac{dzdt}{t^{n+1}}    \Bigg)^{\frac{1}{2}}d\mu\\     
		& \leq\frac{C}{|Q|}\int_{ Q}\sum_{k=1}^{\infty}\Bigg( \iint_{\mathbb{R}_{+}^{n+1}}\left(\frac{t}{|z|+t}\right)^{n\lambda}\Big(\int_{\left(\Delta_{k+2}\right)^{2}\backslash\left(\Delta_{k+1}\right)^{2}}|K_{t}(\mu-z,\vec{y})\\
		& \quad 	-K_{t}(\mu_{0}-z,\vec{y})|^{q}d\vec{y}\Big)^{\frac{2}{q}}\Big(\int_{(2^kQ^*)^{2}}\big(\prod_{i=1}^{2}|b_{i}(y_{i})-\lambda_{i}||f_{i}(y_{i})|\big)^{q'}d\vec{y}\Big)^{\frac{2}{q'}}\frac{dzdt}{t^{n+1}}    \Bigg)^{\frac{1}{2}}d\mu\\     
		& \leq\frac{C}{|Q|}\int_{ Q}\sum_{k=1}^{\infty}C_{k}2^{-\frac{2kn}{q'}}|\mu-\mu_{0}|^{-\frac{2n}{q'}}\left(1+2^{k}|\mu-\mu_{0}|\right)^{-N}\\
		& \quad \times \left(\int_{2^kQ^*}\left(|b_{1}(y_{1})-\lambda_{1}||f_{1}(y_{1})|\right)^{q'}dy_{1}\right)^{\frac{1}{q'}}\left(\int_{2^kQ^*}|b_{2}(y_{2})-\lambda_{2}||f_{2}(y_{2})|^{q'}dy_{2}\right)^{\frac{1}{q'}}d\mu\\
		&\leq \frac{C}{|Q|}\int_Q \sum_{k=1}^{\infty} \frac{C_k |2^kQ^*|^{\frac{2}{q'}}}{(2^kr)^{\frac{2n}{q'}}(1+2^kr)^{N}}\left(\frac{1}{|2^kQ^*|}\int_{2^kQ^*}\left(|b_{1}(y_{1})-\lambda_{1}||f_{1}(y_{1})|\right)^{q'}dy_{1}\right)^{\frac{1}{q'}}\\
		& \quad \times \left(\frac{1}{|2^kQ^*|}\int_{2^kQ^*}|b_{2}(y_{2})-\lambda_{2}||f_{2}(y_{2})|^{q'}dy_{2}\right)^{\frac{1}{q'}}d\mu\\
		&\leq \frac{C}{|Q|}\int_Q \sum_{k=1}^{\infty} \frac{C_k |2^kQ^*|^{\frac{2}{q'}} \varphi(2^kQ^*)^{\frac{2\eta}{l}}}{(2^kr)^{\frac{2n}{q'}}(1+2^kr)^{N}}
		\prod_{i=1}^{2}\left(\frac1{\varphi(2^kQ^*)^{\eta}|2^kQ^*|}
		\int_{2^kQ^*}|f_i(y_i)|^{l} dy_i\right)^{\frac1{l}}\\
		& \quad \times 
		\prod_{j=1}^{2}\left( \frac1{|2^kQ^*|} \int_{2^kQ^*}|b_{j}(y_{j})-\lambda_{j}|^{q'h'} 
		dy_j\right)^{\frac1{q'h'}}d\mu \\
		&\leq C \sum_{k=1}^{\infty} \frac{C_k 
			|2^kQ^*|^{\frac{2}{q'}} \varphi(2^kQ^*)^{\frac{2\eta}{l}}}{(2^kr)^{\frac{2n}{q'}}(1+2^kr)^{N}} \mathcal{M}_{l,\varphi,\eta}(\vec f)(x) k^{2}\| b_1 \|_{BMO_\theta (\varphi)} \| b_2 \|_{BMO_\theta (\varphi)}\varphi(2^kQ^*)^{2\theta}  \\
		&\leq C\| b_1 \|_{BMO_\theta (\varphi)}\| b_2 \|_{BMO_\theta (\varphi)} \mathcal{M}_{l,\varphi,\eta}(\vec f)(x).
		\end {align*}
			\noindent{\it \rm \textbf{Case 2}:} $r\geq1$. 
		Let ${\widetilde Q}^*=8Q$. Splitting each $f_j$ as $f_j=f_j^0+f_j^\infty=f_j \chi_{{\widetilde Q}^*}+f_j \chi_{({\widetilde Q}^*)^{c}}$. Then
		\begin {align*}
		\prod_{j=1}^2f_j(y_j)
		&=\prod_{j=1}^2 f_j^{0}(y_j)+\sum_{(\alpha_1,\alpha_2)\in 
			\mathscr{L}}f_1^{\alpha_1}(y_1) f_2^{\alpha_2}(y_2),
		\end {align*}
		where  $\mathscr{L}:=\{(\alpha_1,\alpha_2)$: there is at least one $\alpha_j=\infty\}$. 
		
		For $\mu\in Q$ and $\widetilde\lambda_j:=(b_j)_{{\widetilde Q}^*}$, $j=1, 2$, we have
		\begin{align*}
			&\left(\frac1{\varphi(Q)^{\eta}|Q|}\int_Q \Big|g_{\lambda,\prod \vec b}^{*}(\vec 
			f)(\mu)\Big|^\delta  d\mu\right)^{\frac1\delta}\\
			&\leq \frac{C}{\varphi(Q)^{\eta/\delta}} \left(\frac{1}{|Q|}\int_{ Q}\Big|\left(b_1(\mu)-\widetilde\lambda_1\right)\left(b_2(\mu)-\widetilde\lambda_2\right)g_{\lambda}^{*}(f_1,f_2)(\mu)\Big|^{\delta}d\mu\right)^{\frac{1}{\delta}}\\
			&\quad +\frac{C}{\varphi(Q)^{\eta/\delta}} \left(\frac{1}{|Q|}\int_{ Q}\Big|\left(b_1(\mu)-\widetilde\lambda_1\right)g_{\lambda,b_{2}}^{*,2}(f_1,f_2)(\mu)\Big|^{\delta}d\mu\right)^{\frac{1}{\delta}}\\
			&\quad +\frac{C}{\varphi(Q)^{\eta/\delta}} \left(\frac{1}{|Q|}\int_{ Q}\Big|\left(b_2(\mu)-\widetilde\lambda_2\right)g_{\lambda,b_{1}}^{*,1}(f_1,f_2)(\mu)\Big|^{\delta}d\mu\right)^{\frac{1}{\delta}}\\
			&\quad +\frac{C}{\varphi(Q)^{\eta/\delta}} \left(\frac{1}{|Q|}\int_{ Q}\Big|g_{\lambda}^{*}\left((b_1-\widetilde\lambda_1)f_1^{0},(b_2-\widetilde\lambda_2)f_2^{0}\right)(\mu)\Big|^{\delta}d\mu\right)^{\frac{1}{\delta}}\\
			&\quad +\sum_{(\alpha_1,\alpha_2)\in\mathscr{L}}\frac{C}{\varphi(Q)^{\eta/\delta}} \left(\frac{1}{|Q|}\int_{ Q}\Big|g_{\lambda}^{*}\left((b_1-\widetilde\lambda_1)f_1^{\alpha_{1}},(b_2-\widetilde\lambda_2)f_2^{\alpha_2}\right)(\mu)\Big|^{\delta}d\mu\right)^{\frac{1}{\delta}}\\
			&:=\widetilde M_{1}+\widetilde M_{2}+\widetilde M_{3}+\widetilde M_{4}+\widetilde M_{5}\\
			&:=\widetilde M_{1}+\widetilde M_{2}+\widetilde M_{3}+\widetilde M_{4}+\sum_{(\alpha_1,\ldots,\alpha_m)\in\mathscr{L}}\widetilde M_{5 \alpha_1,\alpha_2}.
			\end {align*}
			
			Since $\eta>\frac{2\theta}{\frac{1}{\delta}-\frac{1}{\varepsilon}}$, $\widetilde M_{i}$  have the same estimation method of $ M_{i}$, $i=1,2,3,4$. We only estimate $\widetilde M_{5}$. It is easy to get that $ {\textstyle \sum_{j=1}^{2}} 
			|\mu-y_j|\sim 2^{k}r$ for $\mu\in Q$ and $(y_1,y_2)\in (2^{k+3}Q)^2 \setminus (2^{k+2}Q)^2$. Take $N >2\eta/l+2\theta$, by H\"{o}lder's inqualitiy, Minkowski's inequality and the size condition \eqref{1.1}, so
			\begin {align*}	
			\widetilde M_{5 \alpha_1,\alpha_2}
			&\leq\frac C{\varphi(Q)^{\eta/\delta}{|Q|}}\int_{ Q}\Bigg(\iint_{\mathbb{R}_{+}^{n+1}}\left(\frac{t}{|\mu-z|+t}\right)^{n\lambda}\Big(\int_{{({\mathbb{R}^n})^2}\backslash\left(\widetilde Q^{*}\right)^{2}}\left|K_{t}(z,y_{1},y_{2})\right|\\
			&\quad \times\prod_{i=1}^{2}|b_{i}(y_{i})-\widetilde\lambda_{i}|\left|f_{i}(y_{i})\right|dy_{1}dy_{2} \Big)^{2} \frac{dzdt}{t^{n+1}}\Bigg)^{\frac{1}{2}}d\mu\\
			&\leq\frac C{\varphi(Q)^{\eta/\delta}{|Q|}}\int_{ Q}\int_{{({\mathbb{R}^n})^2}\backslash\left(\widetilde Q^{*}\right)^{2}}\Bigg(\iint_{\mathbb{R}_{+}^{n+1}}\left(\frac{t}{|\mu-z|+t}\right)^{n\lambda}|K_{t}(z,y_{1},y_{2})|^{2}\frac{dzdt}{t^{n+1}}\Bigg)^{\frac{1}{2}}\\
			&\quad \times\prod_{i=1}^{2}|b_{i}(y_{i})-\widetilde\lambda_{i}|\left|f_{i}(y_{i})\right|dy_{1}dy_{2} d\mu\\  
			&\leq \frac C{\varphi(Q)^{\eta/\delta}{|Q|}}\int_{ Q}\sum_{k=1}^{\infty}\int_{\left(2^{k+3}Q\right)^{2}\backslash\left(2^{k+2}Q\right)^{2}}\frac{\prod_{i=1}^{2}|b_{i}(y_{i})-\widetilde\lambda_{i}|\left|f_{i}(y_{i})\right|}{\left(\sum_{j=1}^{2}|\mu-y_{j}|\right)^{2n}\left(1+\sum_{j=1}^{2}|\mu-y_{j}|\right)^{N}}dy_{1}dy_{2}d\mu\\
			&\leq \frac C{\varphi(Q)^{\eta/\delta}{|Q|}}\int_{ Q}\sum_{k=1}^{\infty}\frac{|2^{k+3}Q|^{2}}{(2^{k}r)^{2n}(1+2^{k}r)^{N}}\prod_{i=1}^{2}\left(\frac{1}{|2^{k+3}Q|}\int_{2^{k+3}Q}|b_{i}(y_{i})-\widetilde\lambda_{i}|\left|f_{i}(y_{i})\right|dy_{i}\right) d\mu\\
			&\leq \frac C{\varphi(Q)^{\eta/\delta}}\sum_{k=1}^{\infty}\frac{|2^{k+3}Q|^{2}\varphi(2^{k+3}Q)^{\frac{2\eta}{l}}}{(2^{k}r)^{2n}(1+2^{k}r)^{N}}\prod_{i=1}^{2}\left(\frac{1}{\varphi(2^{k+3}Q)^{\eta}|2^{k+3}Q|}\int_{2^{k+3}Q }|f_{i}(y_{i})|^{l}dy_{i}\right)^{\frac{1}{l}}\\	
			&\quad \times\prod_{j=1}^{2}\left(\frac{1}{|2^{k+3}Q|}\int_{2^{k+3}Q}\left|\widetilde\lambda_{j}-b_{j}(y_{j})\right|^{l'}dy_{j}\right)^{\frac{1}{l'}}\\ 
			&\leq \frac C{\varphi(Q)^{\eta/\delta}}\sum_{k=1}^{\infty}\frac{|2^{k+3}Q|^{2}\varphi(2^{k+3}Q)^{\frac{2\eta}{l}}}{(2^{k}r)^{2n}(1+2^{k}r)^{N}}\mathcal{M}_{l,\varphi,\eta}(\vec f)(x)\\
			&\quad \times k^{2}\| b_1 \|_{BMO_\theta (\varphi)}\| b_2 \|_{BMO_\theta (\varphi)}\varphi(2^{k+3}Q)^{2\theta}\\
			&\leq C\sum_{k=1}^{\infty}k^{2}\left(1+2^{k}r\right)^{\frac{2\eta}{l}-N+2\theta}\| b_1 \|_{BMO_\theta (\varphi)}\| b_2 \|_{BMO_\theta (\varphi)}\mathcal{M}_{l,\varphi,\eta}(\vec 
			f)(x)\\
			&\leq C\sum_{k=1}^{\infty}k^{2}\left(2^{k}\right)^{\frac{2\eta}{l}-N+2\theta}\| b_1 \|_{BMO_\theta (\varphi)}\| b_2 \|_{BMO_\theta (\varphi)}\mathcal{M}_{l,\varphi,\eta}(\vec 
			f)(x)\\
			&\leq C\| b_1 \|_{BMO_\theta (\varphi)}\| b_2 \|_{BMO_\theta (\varphi)}
			\mathcal{M}_{l,\varphi,\eta}(\vec f)(x).
			\end {align*}
 
 In conclusion, we have
 \begin{align*} 
 	&M_{\delta,\varphi,\eta}^{\sharp,\triangle}(g_{\lambda,\prod \vec b}^{*}(\vec f))(x)\\
 	&\leq C \prod_{j=1}^{2}\| b_j \|_{BMO_{\theta_{j}}(\varphi)} \left(M_{\varepsilon,\varphi,\eta}^\triangle (g_{\lambda}^{*}(\vec f))(x)+\mathcal{M}_{l,\varphi,\eta}(\vec f)(x)\right)\\
 	&\quad 
 	+C\| b_{1}\|_{BMO_{\theta}(\varphi)}M_{\varepsilon,\varphi,\eta}^\triangle (g_{\lambda,b_{2}}^{*,2}(f_{1},f_{2}))(x)+C\| b_{2}\|_{BMO_{\theta}(\varphi)}M_{\varepsilon,\varphi,\eta}^\triangle (g_{\lambda,b_{1}}^{*,1}(f_{1},f_{2}))(x).
 \end{align*}
 
 The proof of Theorem \ref{Theorem4.1} is finished.
\end{proof}
\subsection{Boundedness of Multilinear Iterative Commutator on weighted Leb-esgue Spaces}
\quad\quad In this subsection, with the help of Theorem \ref{Theorem4.1}, we obtain the boundedness of multilinear iterative commutator on weighted Lebesgue spaces.
\begin{theorem}\label{Theorem4.2} 
Suppose that $m\geq2$, $g_{\lambda}^{*}$ is the new multilinear Littlewood--Paley function with generalized kernel as in Definition \ref {definition1.1}. Let ${\textstyle \sum_{k=1}^{\infty }}  k^m C_k<\infty$, $\vec{\theta}=(\theta_1,\ldots,\theta_m) $, $ \theta_j\geq 0 $, $ j=1,\ldots,m $, $\vec{p}=(p_1,\ldots,p_m)$, $1/p=1/p_1+\cdots+1/p_m$, $\vec \omega=(\omega_1,\ldots,\omega_m)\in A_{\vec{p}/q'}^\infty(\varphi)$, $v_{\vec\omega}= {\textstyle \prod_{j=1}^{m}} \omega_j^{p/p_j}$ and $\vec{b}\in BMO_{\vec\theta}^m(\varphi)$.
If $q'<p_j<\infty$, $j=1,\ldots,m$, then there exists a constant $C>0$ such that
\begin{align*}
	\| g_{\lambda,\prod \vec b}^{*}(\vec f)\|_{L^{p}(v_{\vec\omega})}\le C\prod_{j=1}^{m} \|	b_j \|_{BMO_{\theta_j}(\varphi)}\prod_{j=1}^{m}\| f_j \|_{L^{p_j}(\omega_j)}.\nonumber
\end{align*}
\end{theorem}
\begin{proof}
Take $\eta_0$ and $l$ the same as in the proof of Theorem \ref{Lemma3.2}. Select $\delta,\beta_{1},\ldots,\beta_{m}$ such that $0<\beta_{m}<1/m$, $0<\beta_{j}<\frac{1}{\frac{\sum_{j=1}^{m}\theta_{j}}{\eta_0}+\frac{1}{\beta_{j+1}}}$, $j=1,\ldots,m-1$ and $0<\delta<\frac{1}{\frac{\sum_{j=1}^{m}\theta_{j}}{\eta_0}+\frac{1}{\beta_{1}}}$, then $\eta_0>\frac{\sum_{j=1}^{m}\theta_{j}}{\frac{1}{\delta}-\frac{1}{\beta_{1}}}$ and $\eta_0>\frac{\sum_{j=1}^{m}\theta_{j}}{\frac{1}{\beta_{j}}-\frac{1}{\beta_{j+1}}}$, $j=1,\ldots,m-1$. Applying Theorem \ref{Theorem4.1}, we have
\begin{align*} 
	&\|M_{\delta,\varphi,\eta_{0}}^{\sharp,\triangle}(g_{\lambda,\prod \vec b}^{*}(\vec f))\|_{L^{p}(v_{\vec{\omega}})}\\
	&\leq C \prod_{j=1}^{m}\| b_j \|_{BMO_{\theta_{j}}(\varphi)} \left(\|M_{\beta_{1},\varphi,\eta_{0}}^\triangle (g_{\lambda}^{*}(\vec f))\|_{L^{p}(v_{\vec{\omega}})}+\|\mathcal{M}_{l,\varphi,\eta_{0}}(\vec f)\|_{L^{p}(v_{\vec{\omega}})}\right)\\
	&\quad
	+C\sum_{j=1}^{m-1}\sum_{\xi \in C_{j}^{m}}\prod_{i=1}^{m}\| b_{\xi (i)}\|_{BMO_{\theta_{\xi(i)}}(\varphi)}\|M_{\beta_{1},\varphi,\eta_{0}}^\triangle (g_{\lambda,\prod \vec b_{\xi'}}^{*}(\vec f))\|_{L^{p}(v_{\vec{\omega}})}.
\end{align*}

Acoording to Definition \ref{Definition4.1}, $\xi=\left\{\xi(1),\ldots,\xi(j)\right\}$, $\xi'=\left\{\xi(j+1),\ldots,\xi(m)\right\}$ and $A_{h}=\left\{\xi_{1}:\xi_{1}\subset\xi'\right\}$, where $\xi_{1}$ is any finite subset of $\xi'$ composed of different elements. Denote $\xi_{1}'=\xi'-\xi_{1}$.

Repeated application of the Theorem \ref{Theorem4.1} can obtain
\begin{align*}
	&\|M_{\beta_{1},\varphi,\eta_{0}}^{\sharp,\triangle}(g_{\lambda,\prod \vec b_{\xi'}}^{*}(\vec f))\|_{L^{p}(v_{\vec{\omega}})}\\
	&\leq C \prod_{k=j+1}^{m}\| b_{\xi_{(k)}} \|_{BMO_{\theta_{\xi_{(k)}}}(\varphi)} \left(\|M_{\beta_{2},\varphi,\eta_{0}}^\triangle (g_{\lambda}^{*}(\vec f))\|_{L^{p}(v_{\vec{\omega}})}+\|\mathcal{M}_{l,\varphi,\eta_{0}}(\vec f)\|_{L^{p}(v_{\vec{\omega}})}\right)\\
	&\quad
	+C\sum_{h=1}^{m-j-1}\sum_{\xi_{1} \in A_{h}}\prod_{i=1}^{h}\| b_{\xi_{1} (i)}\|_{BMO_{\theta_{\xi_{1}(i)}}(\varphi)}\|M_{\beta_{2},\varphi,\eta_{0}}^\triangle (g_{\lambda,\prod \vec b_{\xi_{1}'}}^{*}(\vec f))\|_{L^{p}(v_{\vec{\omega}})}.
\end{align*}

By Lemma \ref{Lemma2.3} and Theorem \ref{Theorem3.1}, we continue in this way to obtain	
\begin{align*} 
	&\|M_{\delta,\varphi,\eta_{0}}^{\sharp,\triangle}(g_{\lambda,\prod \vec b}^{*}(\vec f))\|_{L^{p}(v_{\vec{\omega}})}\\
	&\leq C \prod_{j=1}^{m}\| b_j \|_{BMO_{\theta_{j}}(\varphi)} \Big(C_{m+1}(m,n)\|\mathcal{M}_{l,\varphi,\eta_{0}}(\vec f)\|_{L^{p}(v_{\vec{\omega}})}+C_{1}(m,n)\|M_{\beta_{1},\varphi,\eta_{0}}^\triangle (g_{\lambda}^{*}(\vec f))\|_{L^{p}(v_{\vec{\omega}})}\\
	&\quad +C_{2}(m,n)\|M_{\beta_{2},\varphi,\eta_{0}}^\triangle (g_{\lambda}^{*}(\vec f))\|_{L^{p}(v_{\vec{\omega}})}+\ldots+C_{m}(m,n)\|M_{\beta_{m},\varphi,\eta_{0}}^\triangle (g_{\lambda}^{*}(\vec f))\|_{L^{p}(v_{\vec{\omega}})}\Big),
\end{align*}
where $C_{1}(m,n), C_{2}(m,n),\ldots, C_{m}(m,n), C_{m+1}(m,n)$ are all finite real numbers associated with $m$ and $n$.

It follows froms Lemma \ref{Lemma2.3}, Theorem \ref{Theorem 2.1} and Lemma \ref{Lemma2.5}, we have
\begin{align*} 
	&\| g_{\lambda,\prod \vec b}^{*}(\vec f)\|_{L^{p}(v_{\vec\omega})}\leq \|M_{\delta,\varphi,\eta_{0}}^{\triangle}(g_{\lambda,\prod \vec b}^{*}(\vec f))\|_{L^{p}(v_{\vec{\omega}})}\leq C\|M_{\delta,\varphi,\eta_{0}}^{\sharp,\triangle}(g_{\lambda,\prod \vec b}^{*}(\vec f))\|_{L^{p}(v_{\vec{\omega}})}\\
	&\leq C \prod_{j=1}^{m}\| b_j \|_{BMO_{\theta_{j}}(\varphi)} \Big(C_{m+1}(m,n)\|\mathcal{M}_{l,\varphi,\eta_{0}}(\vec f)\|_{L^{p}(v_{\vec{\omega}})}+C_{1}(m,n)\|M_{\beta_{1},\varphi,\eta_{0}}^\triangle (g_{\lambda}^{*}(\vec f))\|_{L^{p}(v_{\vec{\omega}})}\\
	&\quad +C_{2}(m,n)\|M_{\beta_{2},\varphi,\eta_{0}}^\triangle (g_{\lambda}^{*}(\vec f))\|_{L^{p}(v_{\vec{\omega}})}+\ldots+C_{m}(m,n)\|M_{\beta_{m},\varphi,\eta_{0}}^\triangle (g_{\lambda}^{*}(\vec f))\|_{L^{p}(v_{\vec{\omega}})}\Big)\\
	&\leq  C \prod_{j=1}^{m}\| b_j \|_{BMO_{\theta_{j}}(\varphi)}\Big(C_{m+1}(m,n)\|\mathcal{M}_{l,\varphi,\eta_{0}}(\vec f)\|_{L^{p}(v_{\vec{\omega}})}+C_{m+2}(m,n)\|M_{\beta_{m},\varphi,\eta_{0}}^\triangle (g_{\lambda}^{*}(\vec f))\|_{L^{p}(v_{\vec{\omega}})}\Big)\\
	&\leq  C \prod_{j=1}^{m}\| b_j \|_{BMO_{\theta_{j}}(\varphi)}\Big(C_{m+1}(m,n)\|\mathcal{M}_{l,\varphi,\eta_{0}}(\vec f)\|_{L^{p}(v_{\vec{\omega}})}+C_{m+2}(m,n)\|M_{\beta_{m},\varphi,\eta_{0}}^{\sharp,\triangle} (g_{\lambda}^{*}(\vec f))\|_{L^{p}(v_{\vec{\omega}})}\Big)\\
	&\leq  C \prod_{j=1}^{m}\| b_j \|_{BMO_{\theta_{j}}(\varphi)}\Big(C_{m+1}(m,n)\|\mathcal{M}_{l,\varphi,\eta_{0}}(\vec f)\|_{L^{p}(v_{\vec{\omega}})}+C_{m+2}(m,n)\|M_{q',\varphi,\eta_{0}} (\vec f)\|_{L^{p}(v_{\vec{\omega}})}\Big)\\
	&\leq  C \prod_{j=1}^{m}\| b_j \|_{BMO_{\theta_{j}}(\varphi)}\|\mathcal{M}_{l,\varphi,\eta_{0}}(\vec f)\|_{L^{p}(v_{\vec{\omega}})}= C \prod_{j=1}^{m}\|_{BMO_{\theta_{j}}(\varphi)}\|\mathcal{M}_{\varphi,\eta_{0}}(|\vec f|)^{l}\|_{L^{p/l}(v_{\vec{\omega}})}^{1/l}\\
	&\leq C\prod_{j=1}^{m} \|	b_j \|_{BMO_{\theta_j}(\varphi)}\prod_{j=1}^{m}\| f_j \|_{L^{p_j}(\omega_j)}.
\end{align*}

The proof of Theorem \ref{Theorem4.2} is finished.
\end{proof}			
\section*{Statements and Declarations}
\textbf{Confict of interest} No potential confict of interests was reported by the authors.\\
\textbf{Data availability} The manuscript has no associated data.\\
\textbf{Compliance with Ethical Standards }No research involving human participants and
animals.
\section*{Author contributions}
Huimin Sun and Shuhui Yang were responsible for the proof of manuscript Theorems \ref{Theorem 2.1}, \ref{Theorem 2.2}, \ref{Theorem3.1}, \ref{Theorem3.2} and the drafting of the article. Yan Lin was responsible for the proof of the manuscript Theorems \ref{Theorem4.1}, \ref{Theorem4.2}, and made changes to the grammar of the article. All authors reviewed the manuscript.

\bigskip

\noindent  

\smallskip

\noindent School of Science, China University of Mining and
Technology, Beijing 100083,  People's Republic of China

\smallskip

\noindent{\it E-mails:} \texttt{sunhuimin@student.cumtb.edu.cn} (H. Sun)

\noindent\phantom{{\it E-mails:} }\texttt{yangshuhui@student.cumtb.edu.cn} (S. Yang)

\noindent\phantom{{\it E-mails:} }\texttt{linyan@cumtb.edu.cn} (Y. Lin)
\end{document}